\documentclass[12pt]{amsart}

\usepackage{xcolor}
\usepackage[normalem]{ulem} %\sout{...}  to strike out sentences
\RequirePackage[all]{xy}

\newtheorem{thm}{Theorem}[section]
\newtheorem{cor}[thm]{Corollary}
\newtheorem{lem}[thm]{Lemma}
\newtheorem{prop}[thm]{Proposition}
\theoremstyle{definition}
\newtheorem{defn}[thm]{Definition}
\newtheorem{rem}[thm]{Remark}

\newtheorem{ex}[thm]{Example}
\numberwithin{equation}{section}

\theoremstyle{plain}

\usepackage{amsmath}
\usepackage{amsfonts}
\usepackage{amssymb}
\usepackage{color}
\usepackage{mathtools}
\usepackage{bbm}
\newcommand{\be}{\begin{equation}}
	\newcommand{\en}{\end{equation}}
\newcommand{\Lc}{{\mc L}}
\newcommand{\bei}{\begin{itemize}}
	\newcommand{\eni}{\end{itemize}}
\newcommand{\ip}[2]{\langle{#1}|{#2}\rangle}

\newcommand{\C}{\mathfrak{C}}

\numberwithin{equation}{section}

\usepackage{mathtools}

\DeclarePairedDelimiter{\tn}{|||}{|||}

\newcommand{\mb}{\mathbb}
\newcommand{\mc}{\mathcal}
\newcommand{\eul}{\mathfrak}
\newcommand{\A}{\eul A}
\newcommand{\Ao}{{\eul A}_{0}}
\newcommand{\D}{\mc D}
\newcommand{\K}{\mc K}

\newcommand{\IA}{{\mathcal I}_{\Ao}^{\,\YY}(\A)}

\newcommand{\id}{{\sf e }}

\newcommand{\YY}{\mathfrak Y}

\newcommand{\KK}{\mathfrak K}

\newcommand{\B}{{\eul B}}
\newcommand{\M}{\mathfrak M}

\newcommand{\ad}{^{\mbox{\scriptsize $\dag$}}}

\newcommand{\ppi}{\Pi}
\newcommand{\X}{{\mathfrak X}}

\newcommand{\mult}{\,{\scriptstyle \square}\,}

\newcommand{\LDK}{{\mathcal L}\ad(\D,\mathcal{K})}
\newcommand{\LDH}{{\mathcal L}\ad(\D,\H)}
\newcommand{\idop}{{\mb I}}

\newcommand{\LDb}{{\mc L}\ad(\D)_b}

\numberwithin{equation}{section}

\def\H{\mc H}
\def\x{\relax\ifmmode {\mbox{*}}\else*\fi}

\newcommand{\up}{\raisebox{0.7mm}{$\upharpoonright$}}

\newcommand{\vertiii}[1]{{\left\vert\kern-0.25ex\left\vert\kern-0.25ex\left|#1 
					\right\vert\kern-0.25ex\right\vert\kern-0.25ex\right\vert}}
\definecolor{Magenta}{rgb}{1,0,1}

\newcommand{\BibTeX}{B\kern-0.1emi\kern-0.017emb\kern-0.15em\TeX}
\newcommand{\XYpic}{$\mathrm{X\kern-0.3em\raisebox{-0.18em}{Y}}$-$\mathrm{pic}\,$}

\newcommand{\ed}{\end{document}}
\begin{document}

%-------------------------------------------------------------------------
% editorial commands: to be inserted by the editorial office
%
%\firstpage{1} \volume{228} \Copyrightyear{2004} \DOI{003-0001}
%\seriesextraline{This is the Concrete Title of this Book\br H.E. R and S.T.C. W, Eds.}
%\firstpage{1}

	\title[]{Cauchy-Schwarz inequalities for maps in non commutative $L^p$-spaces.}

%----------Author 1
	\author[G. ~Bellomonte]{Giorgia Bellomonte}

%\author[Birkh\"auser \textit{et al.}]%
%{Birkh\"{a}user Publishing Ltd.}
%
\address{Dipartimento di Matematica e Informatica, Universit\`a degli Studi  di Palermo, Via Archirafi n. 34,  I-90123 Palermo, Italy}
\email{giorgia.bellomonte@unipa.it} 

%
%\thanks{This file has been typeset with the option \texttt{draft} to illustrate that feature and its purpose.}

%----------Author 2
\author[S. ~Ivkovi\'{c}]{Stefan Ivkovi\'{c}}

\address{Mathematical Institute of the Serbian Academy of Sciences and Arts, Kneza Mihaila 36, 11000 Beograd, Serbia}
\email{stefan.iv10@outlook.com}
%----------Author 3

\author[C. ~Trapani]{Camillo Trapani}
%\author[]{Rafa\l \ Ab\l amowicz}
\address{Dipartimento di Matematica e Informatica, Universit\`a degli Studi  di Palermo, Via Archirafi n. 34,  I-90123 Palermo, Italy}
\email{camillo.trapani@unipa.it} 

\subjclass{Primary 46K10, 47A07, 16C10. }
\keywords{Cauchy-Schwarz inequality, noncommutative $L^p$-spaces, positive maps, uncertainty relations, von Neumann algebras, numerical radius norm, quasi *-algebras, cyclic *-representations
}
\date{\today}
%\dedicatory{Last Revised:\\ \today}
%%% ----------------------------------------------------------------------
\begin{abstract} 
In this paper, a generalized Cauchy-Schwarz inequalities for positive sesquilinear maps with values in noncommutative $L^p$-spaces for $p>1$ is obtained. Bound estimates for their real and imaginary parts are also provided and, as an application, a generalization of the uncertainty relation in the context of noncommutative $L^2$-spaces is given. Next, a Cauchy-Schwarz inequality for positive sesquilinear maps with values in the space of bounded linear operators from a von Neumann algebra into a $C^*$-algebra equipped with the numerical radius norm is proved. In the same spirit, a new norm on a noncommutative $L^2$-space which generalizes the classical numerical radius norm of bounded linear operators on a Hilbert space is proposed and a Cauchy-Schwarz inequality for positive sesquilinear maps with values in the space of bounded linear operators from a von-Neumann algebra into the noncommutative $L^2$-space equipped with this new norm is proved.These results are used to get representations of general positive linear maps with values into a non-commutative $L^p$-space  and into certain operator spaces  in several different situations. Some concrete examples are also given. 
% These results are used to get representations of general positive linear maps from a unital *-algebra into a non-commutative $L^p$-space, from a unital *-algebra into the space of bounded linear operators on a von-Neumann algebra, and from a unital *-algebra into the space of bounded linear operators from a von-Neumann algebra into a noncommutative $L^2$-space equipped with its standard 2-norm. Some concrete examples are also given.
\end{abstract}
\label{page:firstblob}
%%% ----------------------------------------------------------------------
\maketitle
%%% ----------------------------------------------------------------------
%\tableofcontents
 \section{Introduction} 
 The main topic of the present chapter is the generalized Cauchy-Schwarz inequality for various classes of positive maps (which are not studied in our previously published papers \cite{BDI, BIvT2, BIvT1}) . The basic motivation for this study relies in the representation theory of a variety of structures, such as algebras, modules, etc.  Indeed, it is very well-known the key role that the Cauchy-Schwarz inequality plays in the ordinary representation theory of *-algebras or partial *-algebras, as basic ingredient for the Gelfand-Naimark-Segal representation constructed from a positive linear functional or from a sesquilinear form with values in the complex field. Replacing these latter with positive linear or sesquilinear maps makes sense in view of the construction of more general representations.

 These considerations seem to be one of the main reasons for which, since the famous Kadison-Schwarz inequality for positive maps between $C^*$-algebras \cite{Kadison0}, several variants of the Cauchy-Schwarz inequality for operator valued maps have been studied; see \cite{BIvT2, Bhatia_Davis, CKK, Fujimoto_Seo, Janssens, Jocic, JKL, JL, JL2, KS, Zamani} . Moreover, other inequalities for positive maps between $C^*$-algebras have been proved, as, for instance, in \cite{DM}. In a recent paper \cite{BIvT2} we have proved that some kind of Cauchy-Schwarz-like inequalities hold for positive maps taking values in certain ordered Banach bimodules over *-algebras, in particular, in noncommutative $L^1$-spaces. In the present chapter we obtain Cauchy-Schwarz like inequalities for positive maps on a von Neumann algebra with values in the non commutative $L^p$-spaces, with $p>1$. The approach adopted in \cite{BIvT2} is not applicable if $p>1$. For this reason we develop here some different techniques. The result is a generalized Cauchy-Schwarz inequality where a factor $2$ appears in the right hand side (Proposition \ref{prop:Lp-CS}). However, as proved in Proposition \ref{prop: 3.7}, if $\Phi$ is a positive sesquilinear map from a vector space $\X$ into the noncommutative $L^p$-space, with $p>1$, the {\em proper} Cauchy-Schwarz inequality is satisfied for couples of elements $x,y \in \X$ for which $\Phi(x,y)$ is a normal operator. Proposition \ref{prop: 3.7} and Corollary \ref{cor_3.8} find their motivation in the Kadison-Schwarz inequality for normal elements (see, Remark \ref{rem_3.9}). Furthermore, in Proposition \ref{prop: estimates} we give norm-estimates and bounds for the real and imaginary parts of a positive sesquilinear map with values in the noncommutative $L^2$-space.These bound estimates allow to obtain a generalization of {\em uncertainty relations} in this framework (Proposition \ref{prop: 3.4} and Proposition \ref{prop: 3.7}). In Remark \ref{rem_3.5} a possible physical interpretation of this result is proposed. 
 In Proposition \ref{prop_HilbMod} the Cauchy-Schwarz inequality for positive sesquilinear maps with values in a particular class of ordered Banach bimodules over *-algebras is proved. This class of bimodules includes, as discussed in Example \ref{ex: 3.9}, the duals  of unital $C^*$-algebras.

 Operator inequalities involving the numerical radius norm are an active international research field (see  \cite{BDMP}, and references therein). Motivated by this fact in Section 4 we envisage positive sesquilinear maps taking their values in the space of bounded linear operators acting from a von Neumann algebra into the space of bounded linear operators in Hilbert space equipped with the numerical radius norm and we prove, in Proposition \ref{prop: num rad norm-CS}, the Cauchy-Schwarz inequality (in operator norm). As noted in Remark \ref{rem_4.2} below, this inequality allows us to construct representations of general positive linear maps from a unital *-algebra into the operator algebra of bounded linear operators on a von Neumann algebra.

 Furthermore, we introduce a new norm on the noncommutative $L^2$-space, which generalizes the classical numerical radius norm of the space of bounded linear operators in Hilbert space (see subsection 4.2 and Remark \ref{rem_4.6}) and in Lemma \ref{lem: Lemma 3.1} we show that this is a well-defined norm. Then in Corollary \ref{cor: 4.7} we prove that every positive linear map with values in a noncommutative $L^2$-space satisfies the Cauchy-Schwarz inequality with respect to this new norm. Finally, we consider the space of bounded linear operators from a von Neumann algebra into a noncommutative $L^2$-space equipped with this new norm and we prove that every positive sesquilinear map with values in this operator space satisfies the Cauchy-Schwarz inequality in the operator norm. This result can be used to represent, in a Banach space, positive sesquilinear maps from a unital *-algebra into the space of bounded linear operators from a von Neumann algebra into the noncommutative $L^2$-space, with its standard norm (Remark \ref{rem_4.11}).

 The last section is devoted to applications. By applying the Cauchy-Schwarz inequalities obtained in the previous sections and repeating the procedure of generalized GNS-construction from \cite{BDI, BIvT2, BIvT1}, we obtain  representations in (quasi) Banach spaces of the general positive maps that we considered in the previous sections. Positive maps play in general an important role in quantum physics, in quantum information theory (see \cite{CDPR, DKM, M} and references therein)  and also in operator theory and linear dynamic, see \cite{Ivkovic2026,Petersson,stormer2}, and this was our main motivation for studying representations of these maps. 
 Finally, we notice that the 
Cauchy-Schwarz inequalities  for general positive
sesquilinear maps with values in ordered Banach bimodules or in  operator spaces play an important role for generating  representations and, moreover, produce new related
inequalities for certain classes of completely positive maps, see the discussion in \cite[Section 4]{BIvT2}. At the end of the chapter, we also discuss some concrete examples.
\section{Notation and preliminary results}

    Throughout the paper we will denote by $\B(\mathcal{X},\mathcal{Y})$ the space of bounded linear maps from the normed space $\mathcal{X}$ into the normed space $\mathcal{Y}$. If $\mathcal{X}=\mathcal{Y}$ we will simply write $\B(\mathcal{X})=\B(\mathcal{X},\mathcal{X})$. 
Let $\H$ be a Hilbert space.  If $T\in \B(\H)$ we will  denote by $T^*$ its adjoint, by $N(T)$ and $R(T)$ its null space and the range of $T$, respectively, and by $\|T\|$ its norm. 

% {\color{red}An {\em (orthogonal) projection} is an operator $P\in\B(\H)$ such that $P=P^*=P^2$. A projection $P\in\B(\H)$ is said to be {\em minimal} if whenever $Q\in\B(\H)$ is a projection such that $0\leq Q\leq P$, then either $Q=0$ or $Q=P$.
% A {\em partial isometry} is an operator $A\in\B(\H)$ that maps one (closed) subspace $\D$ of $\H$ (called {\em initial space} of the partial isometry) isometrically onto another one $A(\D)$ (called {\em final space} of the partial isometry) and is such that $\D^\perp\subseteq N(A)$.}  

Let $\D\subset\H$ a dense subspace of $\H$ and a closed operator $A:\D\to\H$, then there exists a partial isometry $Z$ with initial space the closure $\overline{  R(A^*)}$  and final space $R(A)$
  \cite[Vol. II, Ch. 6, Th. 6.1.2]{Kadison} such that $$A=Z(A^*A)^{1/2}=(AA^*)^{1/2}Z$$ this is called the {\em polar decomposition of} $A$. Indeed, if $A=BC$ with $B$ a partial isometry with initial space $R(C)$ and $C$ a positive operator, then $B=Z$ and $C=(A^*A)^{1/2}$.  If $N(A)=N(A^*)=\{0\}$, then $Z$ is a unitary operator.\\
  % {\color{red}It must be observed that even though both $(A^*A)^{1/2}$ and $(AA^*)^{1/2}$ belong to the same $C^*$-algebra containing $A$, the operator $Z$ may not. However, if $A$ belongs to a von Neumann algebra $\M$ and $UV$ is the polar decomposition of $A$, then both $U$ and $V$ are in $\M$. If the operator $A$ is normal, then $V=(A^*A)^{1/2}=(AA^*)^{1/2}$, hence $UV=T=VU$.}\\
  
%   We remind that a \emph{von Neumann algebra} on a Hilbert space $\H$ is a $*$-subalgebra $\M \subseteq \B(\H)$ which contains the identity operator $I$ and \[ \M = \M'', \] where $\M''=(\M')'\supseteq\M$ and $$\M'=\{X\in\B(\H): AX=XA,\,\, \forall A\in\M\}$$ denotes {\em the commutant of $\M$ in $\B(\H)$}.  A (possibly unbounded) operator $A$ is called {\em affiliated with
% a von Neumann algebra $\M$} % (and we will write $A\, \eta\, \M$,)
% if $A$ is closed,
% densely defined and $UA\subseteq AU $ for every unitary operator $U \in\M'$.\\

A {\em trace} on a von Neumann algebra $\M$ on the Hilbert space $\H$ \cite[Chapter 8]{KadisonII} is a linear  function  $\rho:\M^+\to [0,+\infty]$  with $\M^+$ the positive cone in $\M$ ($\rho(X+Y)=\rho(X)+\rho(Y)$ and $\rho(\lambda X)$, $\forall X,Y\in\M^+$, $\forall \lambda\geq0$) such that  $$\rho(X^*X)=\rho(XX^*), \quad \forall X\in\M\,\,\mbox{ (tracial property)}$$ (with the usual convention $0\cdot\infty=\infty$). A trace $\rho$ is
said to be
\begin{itemize}
    \item[] \hspace{-3em} {\em faithful} if $\rho(X)>0$ for any non zero $X\in\M^+,$
  \item[] \hspace{-3em}   {\em semifinite} if for every non zero $X\in\M^+$ there exists some non zero $Y\in\M^+$ with $\rho(Y)<+\infty,$ such that $0\leq Y\leq X$, 
 \item[]  \hspace{-3em}  {\em normal} if $\rho (\sup_i
X_i)= \sup_i\, \rho(X_i)$ for every bounded increasing net $\{X_i\}\subseteq\M^+,$ 
\item[]  \hspace{-3em}  {\em finite} if $\rho(I)<\infty$. 
\end{itemize}

If $\rho$ is normal and finite then $\rho(P)<\infty$ for all projections $P\in\M$. If the centre $\M\cap\M'$ of $\M$ consists of scalar (complex) multiples of the identity $I$, then $\M$ is said a {\em factor}, hence all its central projections are trivial and it cannot be decomposed into a direct sum of smaller von Neumann algebras. The Murray--von Neumann classification divides factors into the following types: factors of {\em type I} are those isomorphic to $\B(\K)$ for some Hilbert space $\K$, they contain minimal projections, factor of {\em type II} are those  with no minimal projections but admitting a faithful normal semifinite trace and factors of {\em type III} those with no nonzero semifinite trace. \\

Let $\M$ be a von Neumann algebra on a Hilbert space $\H$ equipped with a  faithful normal semifinite trace $\rho$ defined on $\M^+$. For every $p\geq1$ we denote by $L^p(\rho)$  the Banach space completion of the *-ideal of $\M$ (see \cite{Nelson1, Nelson2}):
$$\mathcal{J}_p=\{X\in\M: \rho(|X|^p)<\infty\}$$ (where, as usual, $|X|=(X^*X)^{1/2}$) with respect to the norm $$\|X\|_p=\rho(|X|^p)^{1/p},\quad X\in\mathcal{J}_p.$$ We maintain the notation $\rho$ for the natural extension of $\rho$ to $L^p(\rho)$. If $p=\infty$, one defines $L^\infty(\rho)=\M$ and  $\|\cdot\|_\infty=\|\cdot\|$ (the operator norm in $\B(\H)$);then if $\rho$ is finite  it is $L^\infty(\rho)\subset L^p(\rho)$ for every $p\geq 1$.
For $1\le p<\infty$ we denote by  $L^{p}(\rho)$ the noncommutative $L^{p}$-space associated with $\rho$ on $\M$ and consists of all (possibly unbounded) operators $X$ affiliated to $\M$ such that $\|X\|_p<\infty$. If $1<p<\infty$ we write $q$ for the conjugate exponent, $1/p+1/q=1$.\\

Throughout this paper will often use  the noncommutative Hölder’s inequality (see, e.g., \cite[Section 3]{Nelson2}): let 
$\M$ be a von Neumann algebra equipped with a normal, faithful, semifinite trace $\rho$, if $A\in L^p(\rho)$ and $B\in L^q(\rho)$ with $1/p+1/q=1$, then $AB\in L^1(\rho)$ and $$\rho(|AB|)\leq\|A\|_p\|B\|_q.$$ 
%  {\gbc In \cite[Vol II, Ch. 6 Ex. 6.9.12]{Kadison}  if $\M$ is a  von Neumann algebra on a Hilbert space $\H$ equipped with a  normal faithful semifinite trace $\rho$ on $\M_+$, then there is an orthogonal family $\{Q_j :
% j \in J \}$ of nonzero central projections in $\M$ (i.e., for every $j\in J$, $Q_j\neq0$, $Q_j=Q_j^*=Q_j^2$ and $Q_jA=AQ_j$
%  for all operators $A\in\M$), such that $\sum_{
% j\in J} Q_j = I$ and every $Q_j$
% is the sum of an orthogonal family of mutually equivalent finite projections in $\M$.}\\

If $\X$ is a complex vector space, a map $\Phi:\X\times \X\to \mathcal{B}$ (where $\mathcal{B}$ is either $\B(\H)$ or $L^{p}(\rho)$) is called \emph{sesquilinear} if it is linear in the first variable and conjugate-linear in the second. It is \emph{positive} if $\Phi(x,x)\ge 0$ (as an operator or as an element of $L^{p}(\rho)$) for every $x\in \X$.

We will use the duality pairing between $L^{p}(\rho)$ and $L^{q}(\rho)$ (the space of all those operators $B$ in -or affiliated with- $\M$ for which the functional \( A \mapsto \rho(AB) \) is bounded on $L^{p}(\rho)$) given by

\[
\langle A,B\rangle = \rho(AB),\qquad A\in L^{p}(\rho),\ B\in L^{q}(\rho).
\]

Recall the duality formula

\[
\|A\|_{p} = \sup\{\, |\rho(AB)| : B\in L^{q}(\rho),\ \|B\|_{q}\le 1\,\}.
\]
\smallskip

\begin{rem}
    By Hölder's non-commutative inequality and \cite[Corollary 3.4.6]{Dodds} it follows that $$|\rho(AB)|\leq\|AB\|_1\leq\|A\|_p\|B\|_q,\quad\forall A\in L^p(\rho),\forall B\in L^q(\rho).$$ Hence, if $A_n\to A$ in  $L^p(\rho)$ and  $B_n\to B$ in  $L^q(\rho)$ as $n\to \infty$, then $$\rho(A_nB_n)\to\rho(AB),\quad \mbox{ as }n\to\infty.$$ We will use this fact frequently in our proofs.
\end{rem}
 % \section{An inequality for $L^{2}$-valued positive forms  (from CUTS)}\label{eq: case p=2} In this section, a generalized Cauchy-Schwarz inequality  for general positive sesquilinear map into $L^2(\rho)$ is proved. With this inequality at hand we can represent such positive $L^2(\rho)$-valued maps in a quasi-Banach-space.

 %%%%%%%%%%%%%%%%%%%%%%CUTS

\begin{lem}\cite[Lemma 4.8]{BDI}\label{projectors} Let $\M$ be a von Neumann algebra which is a factor of type I or II, and $\rho$ be a semifinite trace on $\M$. Let  $ W\in\M$ such that $W\geq0$ and $W\in L^p(\rho)$. Then there exists a sequence $\{P_n\}_n$ of finite projections in $\M$ such that 
$$\lim_{n\to\infty}\|W(I-P_n)\|_p=0.$$
\end{lem}

\begin{rem}\label{rem: proof lemma 4.8}
    By the proof of \cite[Lemma 4.8]{BDI}, each $P_n= E_W(1/n, \infty)$, where $E_W$ is the spectral measure corresponding to $W\in\M$.
\end{rem}

\begin{lem}\label{lem: rho}
    Let $\rho$ be a faithful, semi-finite normal trace on a von Neumann algebra $\M$ on the Hilbert space $\H$. Let  $p>1$ and $q=\frac{p}{p-1}$. Then, \begin{itemize}
        \item if $A\in L^p(\rho)$ and  $B\in L^q(\rho)$, with $A,B\geq 0$, then $\rho(AB)\geq0$;
        \item if $A\in L^p(\rho)$ and  $B\in L^q(\rho)$, with $A=A^*$ and $B=B^*$, then $\rho(AB)\in \mathbb{R}$.
    \end{itemize} 
\end{lem}
\begin{proof}
    
% We first notice that if $A,B\in L^{2}(\rho)$ with $A=A^{*}$ and $B=B^{*}$, then $\rho(AB)\in \mathbb{R}$.  
% To prove this, it suffices to observe that $\rho(A_{+}B_{+})\ge 0$ whenever $A_{+}B_{+}\in L^{2}(\rho)$ with $A_{+},B_{+}\ge 0$.
Let first $A_+\in L^p(\rho)$ and  $B_+\in L^q(\rho)\cap\mathcal{L}^\infty(\rho)$ with $A_{+},B_{+}\ge 0$. If $P\in L^{\infty}(\rho)$ is any finite projection commuting with $B_{+}$, then by \cite[Lemma 3.1]{OgasYoshi} it is not hard to check that %by the same arguments used in the proof of Lemma \ref{lem: Lemma 3.1} 
%we have
\[
\rho\bigl(A_{+}B_{+}P\bigr)
= \rho(PB_{+}^{1/2}PA_{+}PB_{+}^{1/2}P) \ge 0.
\]

By Lemma \ref{projectors}, we can choose a sequence of finite spectral projections  $\{P_{n}\}$ corresponding to $B_{+}$ such that

\[
\|B_{+} - B_{+}P_{n}\|_{p} \to 0,\quad \mbox{ as }n\to \infty.
\]

Since $\rho(A_{+}B_{+}P_{n})\ge 0$ for all $n$, letting $n\to\infty$ and applying Hölder's inequality yields

\[
\rho(A_{+}B_{+}) \ge 0.
\]
Since this holds, for all $B_{+}\in L^\infty(\rho)\cap L^q(\rho)$, by passing to the limit, we deduce that it also holds for all $\widetilde{B}_+\in L^q(\rho)$, with $\widetilde{B}_+\geq 0$

Now, for arbitrary selfadjoint $A\in L^{p}(\rho)\cap L^\infty(\rho)$ and $B\in L^{q}(\rho)\cap L^\infty(\rho)$, with  $A=A^*$ and $B=B^*$, we can write
\[
A = A_{+} - A_{-}, \qquad B = B_{+} - B_{-},
\]
with $A_{\pm}\in L^{p}(\rho)\cap L^\infty(\rho)$ and $B_{\pm}\in L^{q}(\rho)\cap L^\infty(\rho)$
and $A_{\pm}, B_{\pm}\geq0$.  
Hence, by the above,
\[
\rho(AB)
= \rho(A_{+}B_{+}) - \rho(A_{+}B_{-}) - \rho(A_{-}B_{+}) + \rho(A_{-}B_{-}) \in \mathbb{R}.
\]

Finally, for general selfadjoint $A\in L^{p}(\rho)$ and $B\in L^{q}(\rho)$, choose sequences of selfadjoint elements $\{A_{n}\}\in L^{p}(\rho)\cap L^\infty(\rho)$ and $\{B_{n}\}\in L^{q}(\rho)\cap L^\infty(\rho)$,
\[
A_n\to A\quad \mbox{ in }L^p(\rho),\qquad \,B_n\to B\quad \mbox{ in }L^q(\rho),\quad \mbox{ as }n\to\infty.
\]

Since $\rho(A_{n}B_{n})\in\mathbb{R}$ for all $n$, letting $n\to\infty$ and applying Hölder's inequality gives $\rho(AB)\in\mathbb{R}$.
\end{proof}
 \bigskip
 
 A {\em quasi *-algebra} $(\A, \A_0)$ is a pair consisting of a vector space $\A$ and a *-algebra $\A_0$ contained in $\A$ as a subspace and such that
\begin{itemize}
	\item $\A$ carries an involution $a\mapsto a^*$ extending the involution of $\A_0$;
	\item $\A$ is  a bimodule over $\A_0$ and the module multiplications extend the multiplication of $\A_0$. In particular, the following associative laws hold:
	\begin{equation}\notag \label{eq_associativity}
		(ca)d = c(ad); \ \ a(cd)= (ac)d, \quad \forall \ a \in \A, \  c,d \in \A_0;
	\end{equation}
	\item $(ac)^*=c^*a^*$, for every $a \in \A$ and $c \in \A_0$.
\end{itemize}

The
\emph{identity} or {\it unit element} of $(\A, \A_0)$, if any, is a necessarily unique element $\id\in \A_0$, such that
$a\id=a=\id a$, for all $a \in \A$.

We will always suppose that
\begin{align*}
	&ac=0, \; \forall c\in \A_0 \Rightarrow a=0 \\
	&ac=0, \; \forall a\in \A \Rightarrow c=0. 
\end{align*}
Clearly, both these conditions are automatically satisfied if $(\A, \A_0)$ has an identity $\id$.\\

A  quasi *-algebra $(\A, \A_0)$ is said to be  {\em normed} if $\A$ is a normed space, with a norm $\|\cdot\|$ enjoying the following properties
\begin{itemize}
	\item there exists $\gamma>0$ such that for every $a\in\A$ $$ 
\max\{\|ac\|, \|ca\|\}\leq \gamma \|a\|, \quad\forall c\in \A_0;$$
	\item $\|a^*\|=\|a\|, \; \forall a \in \A;$
	\item  $\A_0$ is dense in $\A[\|\cdot\|]$.
\end{itemize}
If the normed vector space $\A[\|\cdot\|]$ is complete, then $(\A, \A_0)$ is called a  {\em Banach  quasi *-algebra}.
\begin{ex} If $\M$ is von Neumann algebra and $\rho$ is a normal faithful semifinite trace on $\M_+$, the couple $(L^p(\rho), L^p(\rho)\cap L^\infty(\rho))$ is a Banach quasi *-algebra, see \cite[Proposition 5.6.4]{FT_book}. \end{ex} For more details about quasi *-algebras and their properties, see e.g.  \cite{BTT2006, BT2011, BT2023, Frag2} and  the monograph \cite{FT_book}.
\begin{defn}
Let  $\YY$ be a Banach bimodule  over the  *-algebra $\YY_0$. %(with $\YY_0$ equipped with a  not necessarily sub-multiplicative norm). 
We say that $\YY$ is  an {\em ordered Banach bimodule} over $\YY_0$ if
\begin{itemize}
\item[(i)]  $\YY$ is ordered as a vector space; that is,
$\YY$ contains  a (positive) cone $\KK$, i.e., $\KK\subset \YY$ is such that $\KK+\KK\subset \KK$, $\lambda\KK \subset \KK $ for $\lambda\geq 0$ and $\KK\cap (-\KK)=\{0\}$; 
\item[(ii)] $z^* \KK z \subset \KK, \quad \forall z\in \YY_0.$
\end{itemize} 
\end{defn}

Let $\X$ be a vector space and $\YY$ an ordered Banach module over $\YY_0$ with positive cone $\KK$. Let $\Phi$ be a $\YY$-valued positive  sesquilinear  map on   $\X\times\X$  $$\Phi:(x_1,x_2)\in\X\times\X\to\Phi(x_1,x_2)\in\YY;$$ i.e., a map with the properties  \begin{itemize}
	\item[$i)$] $\Phi(x_1,x_1)\in \KK$,
	\item[$ii)$]$\Phi(\alpha x_1+\beta x_2,\gamma x_3)=\overline{\gamma}[\alpha\Phi( x_1,x_3)+\beta \Phi(x_2,x_3)]$,  
\end{itemize}
with $x_1, x_2,x_3 \in\X$ and $\alpha,\beta,\gamma\in\mathbb{C}$. \\
The $\YY$-valued positive  sesquilinear  map $\Phi$ is called {\em faithful} if $$\Phi(x,x)=0_\YY \;\Rightarrow\; x=0.$$

\begin{defn}\label{defn: $C^*$-valued quasi inner product}
Let $\X$ be a vector space and $\YY$ an ordered Banach bimodule over $\YY_0$. A $\YY$-valued   faithful positive  sesquilinear map $\Phi$ on   $\X\times\X$ is said to be a {\em $\YY$-valued quasi inner product} and we often  write $\ip{x_1}{x_2}_\Phi:=\Phi(x_1,x_2)$,  $x_1,x_2\in\X$. If $$\|\Phi(x,y)\|_\YY\le k\|\Phi(x,x)\|_\YY^{1/2}\,\|\Phi(y,y)\|_\YY^{1/2},\quad \forall x,y\in\X$$ for some $k>1$, $\Phi$ is said to satisfy a {\em generalized Cauchy-Schwarz inequality}. If $\Phi$ satisfies the proper Cauchy-Schwarz inequality, that is if
$$\|\Phi(x,y)\|_\YY\le \|\Phi(x,x)\|_\YY^{1/2}\,\|\Phi(y,y)\|_\YY^{1/2},\quad \forall x,y\in\X$$ then $\Phi$ will be called an {\em $\YY$-valued inner product}. 
\end{defn}

 If a $\YY$-valued quasi inner product $\Phi:\X\times\X\to \YY$ satisfies a generalized Cauchy-Schwarz inequality, then it   induces a quasi norm $\|\cdot\|_\Phi$ on $\X$:

\label{vp norm}  \begin{equation*}
\label{eq: norm phi def}
\|x\|_\Phi :=\sqrt{\|\ip{x}{x}_\Phi\|_\YY}=\sqrt{\|\Phi(x,x)\|_\YY},\quad x\in\X,\end{equation*} 
since (see \cite{BIvT1})
\begin{align*} 	
%& \| a\|_\Phi\geq0, \qquad \forall a\in\X \mbox{ and }\| a\|_\Phi=0\Leftrightarrow a=0,\nonumber\\
%& \|\alpha a\|_\Phi=|\alpha|\| a\|_\Phi, \qquad \forall \alpha\in\mathbb{C}, a\in\X,\nonumber\\ 
&   \|x_1+x_2\|_\Phi\leq\sqrt{k}(\| x_1\|_\Phi+\| x_2\|_\Phi), \qquad \forall x_1,x_2\in\X.%\label{eqn:Q3}
\end{align*} \\
The space $\X$ is then a quasi normed space w.r.to the quasi norm $\|\cdot\|_\Phi$. Similarly, an $\YY$-valued inner product induces a norm on $\X$.

 \begin{defn} 
If the complex vector space $\X$ is complete w.r.to the quasi norm $\|\cdot\|_\Phi$, then $\X$ is called a {\em quasi Banach space with $\YY$-valued quasi inner product} or for short a {\em quasi $B_\YY$-space}. 
\end{defn}
If $\Phi$ is not faithful, we can consider the set $$\mathfrak{N}_\Phi=\{x_1\in\X:\, \Phi(x_1,x_2)=0_\YY, \forall x_2\in \X\},$$
which is a subspace of $\X$ (see \cite{BIvT1}).  If a generalized (or the proper) Cauchy-Schwarz inequality holds, then
$$\mathfrak{N}_\Phi=\{x\in\X:\, \Phi(x,x)=0_\YY\}.$$

We denote by $\Lambda_\Phi(x)$ the coset containing $x\in \X$; i.e., $\Lambda_\Phi(x)=x+\mathfrak{N}_\Phi$ and define 
a $\YY$-valued positive  sesquilinear  map on   $\X/\mathfrak{N}_\Phi\times\X/\mathfrak{N}_\Phi$ as follows: 
%\begin{equation*}\ip{\cdot}{\cdot}_\Phi:\X/\mathfrak{N}_\Phi\times\,\X/\mathfrak{N}_\Phi\to\YY\end{equation*}
 \begin{equation}\label{eq: inner pr phi semidef}
\ip{\Lambda_\Phi(x_1)}{ \Lambda_\Phi(x_2)}_\Phi:=\Phi(x_1,x_2),\quad x_1,x_2\in\X.
\end{equation}The associated quasi norm is:\begin{equation}
\label{eq: norm phi semidef}
\|\Lambda_\Phi(x)\|_\Phi:=\sqrt{\|\Phi(x,x)\|_\YY},\quad x\in\X.\end{equation}
 The quotient space $\X/\mathfrak{N}_\Phi=\Lambda_\Phi(\X)$ is a quasi normed space (see \cite{BIvT1}).
We denote by $\widetilde{\X}$ the completion of $(\X/\mathfrak{N}_\Phi,\|\cdot\|_\Phi)$. 

\begin{defn} Let $(\A,\A_0)$ be a quasi *-algebra and $\Phi:\A \times \A\to \YY$ be a  positive sesquilinear  map. The map $\Phi$  is called\begin{itemize}
    \item {\em left-invariant} if
	 $$\Phi(ac,d)=\Phi(c, a^*d), \quad \forall \ a \in \A, \ c,d \in \A_0.$$\end{itemize} If the quasi *-algebra is also normed, $\Phi$ is called \begin{itemize}   \item {\em bounded} if there exists a constant $M>0$ such that $$\|\Phi(a,b)\|_\YY\leq M \|a\|\|b\|, \quad \forall a,b\in\A.$$
\end{itemize} \end{defn}

 \section{Generalized Cauchy-Schwarz inequality for $L^{p}$-valued positive maps} In this section, a generalized Cauchy-Schwarz inequality  for general positive sesquilinear maps into $L^p(\rho)$ is proved. The inequality has been studied in \cite{BIvT2} just in the case $p=1$ (and not for $p>1$). 
 With this inequality at hand we can represent such positive $L^p(\rho)$-valued maps in a quasi $B_\YY$-space (see Section \ref{sec: appl}).
Since $L^p(\rho)$ carries the operator involution, a positive sesquilinear map $\Phi:\X\times \X\to  L^p(\rho)$ is called {\em hermitian} if $\Phi(y,x) =\Phi(x,y)^*$ for every $x,y\in\X$.
%img 481

    %As usual, we will write $y_1\leq y_2$ whenever $y_2-y_1\in\KK$, with $y_1,y_2\in\KK$ and we will sometimes suppose that $\YY$ has an order-preserving norm in the sense that if $y_1\leq y_2$ with $y_1,y_2\in\KK$, then also $\|y_1\|_\YY\leq \|y_2\|_\YY$. 

\begin{prop}\label{prop:Lp-CS}
Let $\M$ be a von Neumann algebra with a faithful normal semi-finite trace $\rho$, let $1<p<\infty$, and let $\X$ be a complex vector space. If $ \Phi:\X\times \X \longrightarrow L^{p}(\rho)$ is a positive sesquilinear map, then  %the Cauchy--Schwarz type inequality holds
\[
\|\Phi(x,y)\|_{p} \le 2\|\Phi(x,x)\|_{p}^{1/2}\,\|\Phi(y,y)\|_{p}^{1/2},\quad \forall x,y\in\X.
\]\end{prop}

\begin{proof}

Fix $x,y \in \X$. If they are such that $\Phi(x,y)=0$ then the thesis follows. Now let $x,y \in \X$ be  such that $\Phi(x,y)\neq0$ and choose a sequence $\{z_n\}_{n\in\mathbb N} \subset L^{\infty}(\rho)\cap L^{p}(\rho)$ such that

\[
z_n \longrightarrow \Phi(x,y) \quad \text{in } L^{p}(\rho)\quad \mbox{ as }n\to\infty.
\]

Let $z_n = u_n |z_n|$ be the polar decomposition of $z_n$ for all $n\in\mathbb{N}$.
Since $z_n \in L^{p}(\rho)$, we have $|z_n|^{p-1} \in L^{q}(\rho)$, with $q=\frac{p}{p-1}$ and by Hölder’s inequality,

\[
z_n |z_n|^{p-1} \in L^{1}(\rho).
\]

Therefore,

\[
\rho(|z_n|^{p}) = \rho\big( u_n^*z_n\, |z_n|^{p-1} \big)
= \rho\big( z_n|z_n|^{p-1}u_n^* \big)
\quad \forall n\in\mathbb N.
\]

Hence,

\begin{align*}
     \big|\|\Phi(x,y)\|_{p}^p-&
 \rho\big( \Phi(x,y)\, |z_n|^{p-1} u_n^* \big) \big|
\\&\le \big|\|\Phi(x,y)\|_{p}^p-\|z_n\|_{p}^p\big|+\big|\rho\big( (\Phi(x,y)-z_n)\, |z_n|^{p-1} u_n^* \big) \big|\\
&\le \big|\|\Phi(x,y)\|_{p}^p-\|z_n\|_{p}^p\big|+\|\Phi(x,y)-z_n\|_p \,\||z_n|^{p-1} u_n^*\|_q\\
&\le \big|\|\Phi(x,y)\|_{p}^p-\|z_n\|_{p}^p\big|+\|\Phi(x,y)-z_n\|_p \,\||z_n|^{p-1}\|_q \\&= \big|\|\Phi(x,y)\|_{p}^p-\|z_n\|_{p}^p\big|+\|\Phi(x,y)-z_n\|_p \,\|z_n\|_p^{p-1}
\end{align*}

Since $z_n \to \Phi(x,y)$ in $L^{p}(\rho)$ as $n\to\infty$, it follows that $\|z_n\|_{p}^{p-1}\leq M$ for all $n\in\mathbb{N}$ and some $M>0$, hence we conclude that

\[
\rho\big( \Phi(x,y)\, |z_n|^{p-1} u_n^* \big)
\longrightarrow \|\Phi(x,y)\|_{p}^{p},\qquad \mbox{ as } n\to\infty.
\]

For each $n\in\mathbb N$, write

\[
|z_n|^{p-1} u_n^* = \xi_1^{(n)} - \xi_2^{(n)} + i\big( \xi_3^{(n)} - \xi_4^{(n)} \big),
\]
where $\xi_j^{(n)} \in L^{q}(\rho)\cap L^{\infty}(\rho) $, $\xi_j^{(n)} \ge 0$ for all $n\in\mathbb{N}$ and $j\in\{1,\dots,4\}$, and such that $\xi_1^{(n)}\xi_2^{(n)}=\xi_3^{(n)}\xi_4^{(n)}=0$ and $\xi_1^{(n)} - \xi_2^{(n)}=Re(|z_n|^{p-1} u_n^*)$, $ \xi_3^{(n)} - \xi_4^{(n)}=Im(|z_n|^{p-1} u_n^*)$, for all $n\in\mathbb{N}$. Then, for all $n\in\mathbb{N}$
\[
\|\xi_1^{(n)} - \xi_2^{(n)}\|_{q}
= \|\xi_1^{(n)} + \xi_2^{(n)}\|_{q}
\leq \|\, |z_n|^{p-1} u_n^* \|_{q}
\]
 and 
\[
\|\xi_3^{(n)} - \xi_4^{(n)}\|_{q}
= \|\xi_3^{(n)} + \xi_4^{(n)}\|_{q}
\leq \|\, |z_n|^{p-1} u_n^* \|_{q}.
\]
Hence we have $$|\rho\big( \Phi(x,y)\, |z_n|^{p-1} u_n^* \big)|\leq \sum_{j=1}^4|\rho\big( \Phi(x,y)\, \xi_j^{(n)}\big)|.$$

Define scalar sesquilinear forms

\[
\varphi_j^{(n)}(a,b) := \rho\big( \Phi(a,b)\, \xi_j^{(n)} \big),
\qquad a,b\in\X,\ j=1,\dots,4.
\]

Since $\Phi$ is positive and $\xi_j^{(n)} \ge 0$, each $\varphi_j^{(n)}$ is a positive  sesquilinear form on $\X$ by Lemma \ref{lem: rho}.
% since $\rho(\Phi(a,b)\xi_j^{(n)})=\rho(\Phi(a,b)(\xi_j^{(n)})^{1/2}(\xi_j^{(n)})^{1/2})=\rho((\xi_j^{(n)})^{1/2}\Phi(a,b)(\xi_j^{(n)})^{1/2})$ for all $a,b\in\X$ and by the assumption that $\Phi$ is positive.

Applying the scalar Cauchy–Schwarz inequality to  $\varphi_j^{(n)}$ for each $j\in\{1,\dots, 4\}$ gives

\begin{align*}
    \sum_{j=1}^{4}
\big| \rho\big( \Phi(x,y)\, \xi_j^{(n)} \big) \big|
&\le
\sum_{j=1}^{4}
(\varphi_j^{(n)}(x,x))^{1/2}\, (\varphi_j^{(n)}(y,y))^{1/2}\\%by the Cauchy Schwarz ineq. applied to the scalar product in R^4
&\le \left(\sum_{j=1}^{4}
(\varphi_j^{(n)}(x,x))^{1/2}\right)^{1/2}\, \left(\sum_{j=1}^{4}
(\varphi_j^{(n)}(y,y))^{1/2}\right)^{1/2}\\&=\left(\rho\big( \Phi(x,x)\, \sum_{j=1}^{4}\xi_j^{(n)} \big)
\right)^{1/2}\,\left(\rho\big( \Phi(y,y)\, \sum_{j=1}^{4}\xi_j^{(n)} \big)
\right)^{1/2}\\% by the Holder ineq \varphi_j^{(n)}(x,x)
%\le \|\Phi(x,x)\|_{p}\, \|\xi_j^{(n)}\|_{q}^{2},
% \qquad
% \varphi_j^{(n)}(y,y)
% \le \|\Phi(y,y)\|_{p}\, \|\xi_j^{(n)}\|_{q}^{2}.
&\leq \|\Phi(x,x)\|_{p}^{1/2}\,\|\Phi(y,y)\|_{p}^{1/2} \left\|\sum_{j=1}^{4}\xi_j^{(n)}\right\|_{q}\\%&\leq \|\Phi(x,x)\|_{p}^{1/2}\,\|\Phi(y,y)\|_{p}^{1/2} \|\sum_{j=1}^{4}\xi_j^{(n)}\|_{q}\\
&\leq \|\Phi(x,x)\|_{p}^{1/2}\,\|\Phi(y,y)\|_{p}^{1/2} (\|\xi_1^{(n)} + \xi_2^{(n)}\|_{q}+ \|\xi_3^{(n)} + \xi_3^{(n)}\|_{q})
\\&\leq 2 \|\Phi(x,x)\|_{p}^{1/2}\,\|\Phi(y,y)\|_{p}^{1/2} \|\, |z_n|^{p-1} u_n^* \|_{q}.
\end{align*}
    by using the Cauchy-Schwarz inequality in $\mathbb{R}^4$ and the Hölder’s inequality.

% \[
% \varphi_j^{(n)}(x,x)
% \le \|\Phi(x,x)\|_{p}\, \|\xi_j^{(n)}\|_{q}^{2},
% \qquad
% \varphi_j^{(n)}(y,y)
% \le \|\Phi(y,y)\|_{p}\, \|\xi_j^{(n)}\|_{q}^{2}.
% \]}

Since \[
    \|\, |z_n|^{p-1} u_n^* \|_{q}
\leq \|\, |z_n|^{p-1} \|_{q}\|u_n^* \|_{\infty}\leq \| |z_n|^{p-1}\|_{q}=\|z_n\|_p^{p-1},
\]
 it is 

\begin{align}\label{eq: ineq 1}
\big| \rho\big( \Phi(x,y)\, |z_n|^{p-1} u_n \big) \big|
&\nonumber\le
2\, \|\Phi(x,x)\|_{p}^{1/2}\, \|\Phi(y,y)\|_{p}^{1/2}\,
\|\, |z_n|^{p-1} u_n^* \|_{q}\\&\le
2\, \|\Phi(x,x)\|_{p}^{1/2}\, \|\Phi(y,y)\|_{p}^{1/2}\,
\| z_n\|_{p}^{p-1}.
\end{align}

and letting $n\to\infty$ yields

\[
\|\Phi(x,y)\|_{p}^{p}
\le
2\, \|\Phi(x,x)\|_{p}^{1/2}\, \|\Phi(y,y)\|_{p}^{1/2}\,
\|\Phi(x,y)\|_{p}^{p-1}.
\]

Dividing both sides by $\|\Phi(x,y)\|_{p}^{p-1}\neq0 $ gives

\[
\|\Phi(x,y)\|_{p}
\le
2\, \|\Phi(x,x)\|_{p}^{1/2}\, \|\Phi(y,y)\|_{p}^{1/2}.
\]
\end{proof} It worths notice that, when $p=2$, the constant $2$ in the inequality can be reduced to $\sqrt{2}$.
\begin{rem}
    
Let \(p = 2\). If \(A \in L^2(\rho) \cap L^{\infty}(\rho)\), then

\begin{align*}
    \rho(AA^{*})
&= \rho\big( (\Re A - i \Im A)(\Re A + i \Im A) \big)
\\&= \rho\big( (\Re A)^{2} \big) + \rho\big( (\Im A)^{2} \big)
  + i\big( \rho(\Re A \, \Im A) - \rho(\Im A \, \Re A) \big).
\end{align*}

Since \(\Re A, \Im A \in L^2(\rho)\), %(since $\Re A = \tfrac{1}{2}(A + A^{*})$, $\Im A = \tfrac{1}{2i}(A - A^{*})$ and \(A, A^{*} \in L^2(\rho)\)), 
we have
\[
\rho(\Re A \, \Im A) = \rho(\Im A \, \Re A).
\]
Hence, $\rho(AA^*)=\rho((\Re A)^2)+\rho((\Im A)^2)$. By applying this equality in the proof of Proposition \ref{prop:Lp-CS} we deduce that for every $n\in\mathbb{N}$ \begin{align*}
    \||z_n|^{p-1}u_n^*\|_2&=\||z_n|u_n^*\|_2=(\|\xi_1^{n}-\xi_2^{n}\|_2^2+\|\xi_3^{n}-\xi_4^{n}\|_2^2)^{1/2}\\&(\|\xi_1^{n}+\xi_2^{n}\|_2^2+\|\xi_3^{n}+\xi_4^{n}\|_2^2)^{1/2}
\end{align*}
since $\xi_1^{n}\xi_2^{n}=\xi_3^{n}\xi_4^{n}=0$ for every $n\in\mathbb{N}$.
However, by the Cauchy-Schwarz for the inner product in $\mathbb{R}^2$, we get that for every $n\in\mathbb{N}$ \begin{align*}   (\|\xi_1^{n}+\xi_2^{n}\|_2^2+\|\xi_3^{n}+\xi_4^{n}\|_2^2)^{1/2}%&=\ip{\left(\|\xi_1^{n}+\xi_2^{n}\|_2,\|\xi_3^{n}+\xi_4^{n}\|_2\right)}{(1,1)}\\
   &\leq \sqrt{2}\left(\|\xi_1^{n}+\xi_2^{n}\|_2^2+\|\xi_3^{n}+\xi_4^{n}\|_2^2\right)^{1/2}\\&=\sqrt{2}\||z_n|u_n^*\|_2.
\end{align*}\end{rem}

% Therefore,

% \[
% \rho(AA^{*}) = \rho\big( (\Re A)^{2} \big) + \rho\big( (\Im A)^{2} \big).
% \]

%     {\color{red} In alternative:\begin{proof}
% Fix $x,y\in\X$ and set $a=\Phi(x,y)\in L^{p}(\rho)$.
% For any $b\in L^{q}(\rho)$ with $\|b\|_{q}\le 1$, consider the scalar sesquilinear form

% \[
% \varphi_b(u,v)=\rho\!\left(b^{*}\,\Phi(u,v)\,b\right),\qquad u,v\in\X.
% \]

% Since $\Phi$ is positive, $\varphi_b$ is a positive scalar sesquilinear form.
% Hence the scalar Cauchy--Schwarz inequality gives

% \[
% |\varphi_b(x,y)|
% \le \varphi_b(x,x)^{1/2}\,\varphi_b(y,y)^{1/2}.
% \]

% We estimate the diagonal terms using Hölder:

% \[
% \varphi_b(x,x)
% = \rho\!\left(b^{*}\Phi(x,x)b\right)
% \le \|\Phi(x,x)\|_{p}\,\|b\|_{q}^{2}
% \le \|\Phi(x,x)\|_{p},
% \]

% and similarly

% \[
% \varphi_b(y,y)\le \|\Phi(y,y)\|_{p}.
% \]

% Thus

% \[
% |\rho(b^{*}ab)|
% =|\varphi_b(x,y)|
% \le \|\Phi(x,x)\|_{p}^{1/2}\,\|\Phi(y,y)\|_{p}^{1/2}.
% \]

% Since $b^{*}ab$ and $ab$ have the same trace dual pairing norm, we have

% \[
% |\rho(ab)|\le |\rho(b^{*}ab)|.
% \]

% Taking the supremum over all $b\in L^{q}(\rho)$ with $\|b\|_{q}\le 1$ yields

% \[
% \|a\|_{p}
% \le \|\Phi(x,x)\|_{p}^{1/2}\,\|\Phi(y,y)\|_{p}^{1/2}.
% \]

% \end{proof}
% }

\subsection{Estimates for the real and imaginary parts of positive $L^{2}(\rho)$-valued sesquilinear maps}

Motivated by \cite[Lemma 1]{Janssens}, we now provide norm estimates for the real and imaginary parts of positive $L^{2}(\rho)$-valued sesquilinear maps.

\begin{prop}\label{prop: estimates}
Let $\Phi : \X \times \X \to L^{2}(\rho)$ be a positive sesquilinear map. Then, for all $x,y \in \X$,
\[
\| \Re \Phi(x,y) \|_{2}^{2} \le \|\Phi(x,x)\|_{2}\, \|\Phi(y,y)\|_{2},
\quad
\| \Im \Phi(x,y) \|_{2}^{2} \le \|\Phi(x,x)\|_{2}\, \|\Phi(y,y)\|_{2},
\]

where

\[
\Re \Phi(x,y)=\frac{\Phi(x,y)+\Phi(x,y)^{*}}{2},
\qquad
\Im \Phi(x,y)=\frac{\Phi(x,y)-\Phi(x,y)^{*}}{2i}.
\]

\end{prop}

\begin{proof} Fix any $x,y\in\X$. 
We only prove the first inequality concerning $\Re \Phi(x,y)$; the proof of the second one is analogous.
By  Lemma \ref{lem: rho} it is
\[
\Re \left(\rho\bigl(\left(\Re \Phi(x,y)\right)\,\Phi(x,y)\bigr)\right)
= \rho\bigl(\Re \Phi(x,y)\,\Re \Phi(x,y)\bigr)
= \|\Re \Phi(x,y)\|_{2}^2.
\]
Indeed, \[
\rho\bigl(\left(\Re \Phi(x,y)\right)\,\Phi(x,y)\bigr)
= \rho\bigl(\left(\Re \Phi(x,y)\right)^2\bigr)
+ i\,\rho\bigl(\left(\Re \Phi(x,y)\right)\,\left(\Im \Phi(x,y)\right)\bigr),
\] and by the above, both $\rho\left(\bigl(\Re \Phi(x,y)\bigr)^2\right)$ 
and 
$\rho\left(\Re \Phi(x,y)\Im \Phi(x,y)\right)$
are real, so $$\Re\left(\rho\bigl(\left(\Re \Phi(x,y)\right)\,\Phi(x,y)\bigr)\right)=\rho\left(\bigl(\Re \Phi(x,y)\bigr)^2\right).$$

Let $\{z_{n}\}\subset L^{2}(\rho)\cap L^{\infty}(\rho)$ be such that $z_{n}\to \Phi(x,y)$ in $L^{2}(\rho)$ as $n\to \infty$.  
Then

\[
\Re z_{n}=\frac{z_{n} + z_{n}^*}{2} \to \Re \Phi(x,y)=\frac{\Phi(x,y) + (\Phi(x,y))^*}{2} \quad \text{in } L^{2}(\rho), \mbox{ as }n\to \infty.
\]

Write for every $n\in\mathbb{N}$

\[
\Re z_{n} = z_{n}^{+} -z_{n}^{-}, \quad \mbox{ where } z_{n}^{\pm}\ge 0\qquad \mbox{ and }z_{n}^{+}z_{n}^{-}=0.
\]

Then
\[
\bigl|\Re \left(\rho(\Phi(x,y)\,\Re \Phi(x,y))\right)\bigr|
\le\bigl| \rho\bigl(\Phi(x,y)\left(\Re\Phi(x,y)\right)\bigr)\bigr|,
\]
and, on the other hand, 
\begin{align*}
    \bigl| \rho(\Phi(x,y)(z_{n}^{+} - z_{n}^{-}))\bigr|
&\le\|\Phi(x,x)\|_2^{1/2}\|\Phi(y,y)\|_2^{1/2}\|z_{n}^{+} +z_{n}^{-}\|_2\\&=\|\Phi(x,x)\|_2^{1/2}\|\Phi(y,y)\|_2^{1/2}\|z_{n}^{+} - z_{n}^{-}\|_2
\end{align*}
as shown in the proof of Proposition \ref{prop:Lp-CS}.
Letting $n\to\infty$ on both sides yields \begin{align*}  \| \Re\left( \Phi(x,y) \right) \|_2^2&\leq\bigl| \rho\!\left( \Phi(x,y)\, \Re \Phi(x,y) \right) \bigr| \\&\le \|\Phi(x,x)\|_{2}^{1/2}\, \|\Phi(y,y)\|_{2}^{1/2}\, \|\Re \Phi(x,y)\|_{2}. \end{align*}This completes the proof.\end{proof}

Let us consider the partial *-algebra $\LDH$ of all closable operators $A$ such that $D(A)=\D$, $D(A^*)\supset \D$. The involution is defined by $A^\dag:=A^*_{\upharpoonright \D}.$

Let $\Phi: \LDH\times \LDH\to L^2(\rho)$ be a positive sesquilinear map.
Let $A,B\in \LDH$. We say that $A,B$ admit a $\Phi$-commutator, if there exists $\C\in \in \LDH$ such that
$$ \Phi(AX,B^\dag Y)-\Phi(BX, A^\dag Y)= \Phi (CX,Y), \quad \forall X,Y\in \LDb$$ where $\LDb$ denotes the *-algebra of bounded operators of $\LDH$ such that $A\D\subset \D$, $A^\dag \D\subset \D.$
In particular, if $A=A^\dag, B=B^\dag$ are symmetric operators and a $\Phi$-commutator $C$ exists, then $C=iK$ with $K=K^\dag$.

If $A\in \LDH$ is a symmetric operator (i.e., $A=A^\dag$)
we define a map $(\Delta A)_\Phi: {\mb R} \to {\mb R}^+$ by
$$ (\Delta A)_\Phi(\lambda) = \|\Phi(A-\lambda I, A-\lambda I)\|^{1/2}_2, \quad \lambda \in {\mb R}\}.$$
which can be interpreted as a sort of {\em uncertainty} of $A$.

Proposition \ref{prop: estimates} allows us to get an uncertainty relation for certain pairs of symmetric operators. 
\begin{prop}\label{prop: 3.4}
	Let $A,B,K\in \LDH$ be symmetric operators satisfying
\begin{itemize}
	\item[(a)]$\Phi(AX, Y)=\Phi(X,A Y)$, $\Phi(BX, Y)=\Phi(X,B Y)$ for every $X,Y\in \LDb$.
	\item[(b)] $ \Phi(AX,B Y)-\Phi(BX, A Y)= \Phi (iKX,Y), \quad \forall X,Y\in \LDb$
\end{itemize}	
	Then, for every $\lambda, \mu \in {\mb R}$,  the following uncertainty relation holds
\begin{equation}
(\Delta A)_\Phi(\lambda) (\Delta B)_\Phi(\mu) \geq \frac12 \gamma_\Phi (K), 
\end{equation}
where $\gamma_\Phi(K)= \|\Phi(K,I)\|_2$
	\end{prop}
\begin{proof}
As one can prove by a direct easy computation, the condition (a) implies that if $A,B$ satisfy condition (b) then the same holds true for $A-\lambda I$ and $B-\mu I$, with the same operator $K$ on the right hand side. 
Then, from (b) one obviously has
\begin{equation} \label{eqn_commutator} \Phi(A-\lambda I,B-\mu I )-\Phi(B-\mu I, A-\lambda I)= \Phi (iK,I).\end{equation}
By Proposition \ref{prop: estimates}, we obtain
$$\|\Re \Phi(iK,I)\|_2\leq2 \|\Phi(A-\lambda I, A-\lambda I)\|_2^{1/2} \|\Phi(B-\mu I, B-\mu I)\|_2^{1/2} $$
and
$$\|\Im \Phi(iK,I)\|_2\leq 2\|\Phi(A-\lambda I, A-\lambda I )\|_2^{1/2} \|\Phi(B-\mu I, B-\mu I)\|_2^{1/2}. $$
Taking into account \eqref{eqn_commutator} we can prove that $\Phi(K,I)$ is self-adjoint. Then, 
%\footnote{FOR US: we need to use the hermiticity of $\Phi$ and take adjoint in %\eqref{eqn_commutator}},
$$\|\Phi(K,I)\|_2\leq 2 \|\Phi(A-\lambda I, A-\lambda I)\|_2^{1/2} \|\Phi(B-\lambda I, B-\lambda I)\|_2^{1/2} .$$
In other terms
$$\|\Phi(K,I)\|_2\leq 2 (\Delta A)_\Phi(\lambda) (\Delta B)_\Phi(\mu) .$$
\end{proof}
% \begin{rem}
%     It worths notice that, with an obvious modification of the definitions, the statement holds also for $\Phi: \M\times \M\to L^2(\rho)$.
% \end{rem}
\begin{rem}\label{rem_3.5} The interest of the previous proposition relies on its close analogy with the uncertainty relations that in quantum physics is linked to the commutation relations. The term $\gamma_\phi(K)$ can be interpreted as sort of {\em mean value} of $K$ with respect to $\Phi$, while a term like $(\Delta A)_\Phi$ can be understood as the {\em variance} of $A$ with respect to $\Phi$.
\end{rem} 

A statement analogous to Proposition \ref{prop: 3.4} can be obtained in a more abstract setting. 
Let $(\A, \Ao)$ be a quasi *-algebra with unit $\id$ and $\Phi:\A\times \A \to L^2(\rho)$ a positive sesquilinear map satisfying the {\em invariance} condition
$$\Phi(ax,y)=\Phi(x,a^*y), \quad \forall a\in \A, \, x,y \in \Ao.$$
As before, if $a,b\in \A$, we say that $a,b$ admit a $\Phi$-commutator, if there exists $c\in \A$ such that
$$ \Phi(ax,b^* y)-\Phi(bx, a^* y)= \Phi (cx,y), \quad \forall x,y\in \A_0.$$ 
For $a\in \A$ and $\lambda \in {\mb R}$, we put $$\Delta_\Phi(a)(\lambda)=\|\Phi(a-\lambda\id, a-\lambda \id)\|_2.$$
\begin{prop}\label{prop: 3.6} Let $a,b$ be symmetric elements of $\A$ and assume that they admit a $\Phi$-commutator $c$. Then, for every $\lambda, \mu \in {\mb R} $,
\begin{equation}
	(\Delta a)_\Phi(\lambda) (\Delta b)_\Phi(\mu) \geq \frac12 \gamma_\Phi (c), 
\end{equation}
where $\gamma_\Phi(c)= \|\Phi(c,\id)\|_2$.
\end{prop}

\subsection{Cauchy-Schwarz inequalities for normal elements in $L^p(\rho)$}%{Variant for $L^{p}$-valued maps into a von Neumann algebra}

% A special case of Proposition \ref{prop: num rad norm-CS} is the following.

% The same argument as in Proposition \ref{prop:Lp-CS} applies when the $\Phi$ takes its values in a von Neumann algebra $\M$ and one measures the $L^{p}$-norm with respect to a finite faithful normal trace $\rho$.

Let $\rho$ be a faithful, normal, semifinite trace on a von Neumann algebra $\M$.
We have the following proposition.

\begin{prop}\label{prop: 3.7}
	Let $p>1$ and let $\Phi : \X\times \X \to L^{p}(\rho)$ be a positive sesquilinear map.
	Suppose that $\Phi(x,y)\in L^{p}(\rho)\cap L^{\infty}(\rho)$ for some $x,y\in \X$.
	If $\Phi(x,y)$ is normal, then	
	\[
	\|\Phi(x,y)\|_{p}
	\;\le\;
	\|\Phi(x,x)\|_{p}^{1/2}\, \|\Phi(y,y)\|_{p}^{1/2}.
	\]	
\end{prop}
\begin{proof}
	If $x,y\in\X$ are such that  $\Phi(x,y)\in L^{p}(\rho)\cap L^{\infty}(\rho)$ and  $\Phi(x,y)=0$ the inequality holds true, so assume $\Phi(x,y)\neq 0$.
	Let $V_{x,y}$ denote the partial isometry in the polar decomposition of $\Phi(x,y)$
		\[
	\Phi(x,y)=V_{x,y}\,|\Phi(x,y)|.
	\]	
	Let	
	\[
	\Pi=\sum_{j=1}^{N} c_{j} P_{j}
	\]	
	be a simple operator-valued function in $L^q(\rho)\cap L^{\infty}(\rho)$, where $P_{j}$'s are pairwise orthogonal projections in $L^q(\rho)\cap L^\infty(\rho)$. Then,	
	\begin{align}\label{eq: pag 15}
	&|\rho(\Phi(x,y)\Pi)|
	\le
	\sum_{j=1}^{N} |c_{j}|\, |\rho(\Phi(x,y)P_{j})|\\
	\nonumber&\leq\sum_{j=1}^{N} |c_{j}|\, \left(\rho\left(\Phi(x,x) P_{j}\right)\right)^{1/2} \left(\rho\left(\Phi(y,y) P_{j}\right)\right)^{1/2} 
	    	\\\nonumber&
	\le\left(\sum_{j=1}^{N} |c_{j}|\rho\left(\Phi(x,x) P_{j}\right)\right)^{1/2}\left(\sum_{j=1}^{N} |c_{j}|\rho\left(\Phi(y,y) P_{j}\right)\right)^{1/2}\\\nonumber\nonumber&=\left(\rho\left(\Phi(x,x) \left(\sum_{j=1}^{N} |c_{j}|P_{j}\right)\right)\right)^{1/2}\left(\rho\left(\Phi(y,y) \left(\sum_{j=1}^{N} |c_{j}|P_{j}\right)\right)\right)^{1/2}\\\nonumber&=\left(\rho\left(\Phi(x,x) \left(|\Pi|\right)\right)\right)^{1/2}\left(\rho\left(\Phi(y,y) \left(|\Pi|\right)\right)\right)^{1/2}
	\end{align}
	where in the second inequality we have applied the Cauchy-Schwarz inequality to the sesquilinear forms $\varphi_j:\X\times\X\to \mathbb{C}$, $j\in\{1,\dots, N\}$, given by $$\varphi_j(x,y)=\rho(\Phi(x,y)P_j),\quad\forall x,y\in\X$$which are positive by the positivity of $\Phi$ and by Lemma \ref{lem: rho}, whereas in the last inequality we applied the Cauchy-Schwarz inequality  to the inner product in $\mathbb{R}^N$ and in the last equality the fact that $|\Pi|=\sum_{j=1}^{N} |c_{j}|P_{j}$ since $P_iP_j=0$ whenever $i\neq j$.
	
	Now, by the H\"older's inequality, we get

		\begin{align}
	    \label{eq: pag 16}
	\left(\rho\left(\Phi(x,x) \left(\Pi\right)\right)\right)^{1/2}&\left(\rho\left(\Phi(y,y) \left(\Pi\right)\right)\right)^{1/2}
	\\\nonumber&\le
	\|\Phi(x,x)\|_{p}^{1/2}\, \|\Phi(y,y)\|_{p}^{1/2}\,
	\|\Pi\|_{q}.
	\end{align}

	Now assume that $\Phi(x,y)$ is normal.  
	Then

	\[
	|\Phi(x,y)|^{p-1} V_{x,y}^*
	=
	\int_{\sigma(\Phi(x,y))} \mbox{sgn}(z)|z|^{p-1}\, dE(z)\, %V_{x,y},
	\]

	where $E$ is the spectral measure of $\Phi(x,y)$  %$|\Phi(x,y)|$.
	and $\mbox{sgn}:\sigma(\Phi(x,y))\to\mathbb{C}$ is given by $$\mbox{sgn}(z)=\begin{cases}\frac{z}{|z|}, \quad \mbox{if }z\neq0\\0, \quad \mbox{if }z=0.\end{cases}$$
	By Lemma \ref{projectors}, for each $m\in\mathbb{N}$ there exists a spectral projection $P_{m}$ of $|\Phi(x,y)|$ such that
		\[
	\||\Phi(x,y)|(I-P_{m})\|_{p} \leq \frac{1}{m}.
	\]	
	% By spectral functional calculus, since $P_{m}$ commutes with $V_{x,y}$ because $\Phi(x,y)$ is normal, we have

    Now, by Remark \ref{rem: proof lemma 4.8}, we have $P_m=F((\frac{1}{m}, \infty))$, where $F$ is the spectral measure corresponding to $|\Phi(x,y)|$. Since $t^p \chi_{[0,\frac{1}{m}]}(t)=(t\chi_{[0,\frac{1}{m}]}(t))^p$ for all $t\in\sigma(|\Phi(x,y)|)$, by spectral functional calculus we have
	\[
|\Phi(x,y)|^{p}  (I-P_{m})^p=\left(	|\Phi(x,y)|(I-P_{m})
	\right)^p
	\]
	and since $(I-P_{m})^p=I-P_{m}=(I-P_{m})^q$, we obtain:
	\[
	\||\Phi(x,y)|^{p-1}(I-P_{m})\|_{q}
	=
	\||\Phi(x,y)|(I-P_{m})\|_{p}^{p-1}
	< m^{-(p-1)}.
	\]

	Let $\{\Pi_{n}\}$ be a sequence of simple functions on $\sigma(\Phi(x,y))$ converging in supremum norm to the function	
	\[
	z\to \mbox{sgn}(z)|z|^{p-1},
	\]	
	then $\{\int_{\sigma(\Phi(x,y))}\Pi_{n}dE\}_{n}$ is a sequence of operator-valued functions converging to $|\Phi(x,y)|^{p-1}V_{x,y}^*$ in the operator norm. 
    	Let now for each $n$ consider $\{A_1,\dots,A_N\}$ a set of pairwise disjoint Borel subsets of $\sigma(\Phi(x,y))$ with $\sigma(\Phi(x,y))=\sum_{j=1}^NA_j$ such that $$\int_{\sigma(\Phi(x,y))}\Pi_ndE=\sum_{j=1}^Nc_jE(A_j)$$ for some scalars $c_1,\dots,c_n$
 (note that not only  $\{A_1,\dots,A_N\}$ but also $N$ and the scalars $\{c_1,\dots, c_N\}$    depend on $n$)	
	
	Now, $P_m$ mutually commutes with $E(A_j)$ for each $j$ since $P_m$ is a spectral projection corresponding to $|\Phi(x,y)|$ and $E$ is the spectral measure corresponding to $\Phi(x,y)$ which is normal and hence commutes with $|\Phi(x,y)|$.	
%	\[
%	\Pi_{n} P_{m} = \sum_{j=1}^{N_{n}} c_{j}^{(n)}\, E(A_{j}^{(n)}) P_{m}
%	\]	
	Therefore, since the product of two mutually commuting orthogonal  projections is an orthogonal projection too (see. e.g., \cite[Theorem 2.8.4]{Birman_Solomjak}) we get that, for every $n\in\mathbb{N}$ the operator  $\left(\int_{\sigma(\Phi(x,y))}\Pi_{n}dE\right)P_{m}$	is also a simple operator-valued function in $\M$ because 	$$\left(\int_{\sigma(\Phi(x,y))}\Pi_{n}dE\right)P_{m}=\sum_{j=1}^Nc_jE(A_j)P_m$$
	and $E(A_j)P_m$ is orthogonal projection for every $j\in\{1,\dots, N\}$; moreover, for all $i\neq j$, with $i,j\in\{1,\dots, N\}$,
    $$E(A_j)P_mE(A_i)P_m=P_mE(A_j)E(A_i)P_m=0.$$
Furthermore, $P_{m}$ commutes with $V_{x,y}^*$.  
Indeed, since  $\Phi(x,y)$ is normal then $V_{x,y}^*$ commutes with $|\Phi(x,y)|$  and since 
$P_{m}$ is a spectral projection corresponding to $|\Phi(x,y)|$, we must have
\[
P_{m} V_{x,y}^* = V_{x,y}^* P_{m}.
\]
Now observe that since $P_{m}$ is a finite projection (i.e.\ $\rho(P_{m})<\infty$) and hence $P_{m}\in L^q(\rho)$,
\begin{align*}
&\left\|\left(\int_{\sigma(\Phi(x,y))}\Pi_{n}dE\right)P_{m}- |\Phi(x,y)|^{p-1} V_{x,y}^*P_{m} \right\|_{q}\\&\leq \left\|\int_{\sigma(\Phi(x,y))}\Pi_{n}dE- |\Phi(x,y)|^{p-1} V_{x,y}^*\right\|_{\infty}\bigl\|P_{m} \bigr\|_{q}  \to 0,
\end{align*}
    Therefore,
\begin{align*}
\|\, |\Phi(x,y)|^{p-1} V_{x,y}^* (P_{m}-I) \|_{q}&
=
\|\, |\Phi(x,y)|^{p-1} (P_{m}-I) V_{x,y}^* \|_{q}\\&
\le
\|\, |\Phi(x,y)|^{p-1}(P_{m}-I) \|_{q}\|V_{x,y}^* \|_{\infty}\\
&\le
\|\, |\Phi(x,y)|^{p-1}(P_{m}-I) \|_{q}
< m^{-(p-1)}
\end{align*}
 as it has been shown earlier.

\medskip

Since $\left(\int_{\sigma(\Phi(x,y))}\Pi_{n}dE\right)P_{m}$ is a simple operator-valued function in
$L^{q}(\rho)\cap L^{\infty}(\rho)$ for each $n$  and $E(A_j)P_m\in L^q(\rho)$ for all $j\in\{1,\dots, N\}$, from \eqref{eq: pag 15} and \eqref{eq: pag 16}  we obtain

\begin{align}\label{eq: estimates}
&\left|\rho\left(\Phi(x,y)\, \left(\int_{\sigma(\Phi(x,y))}\Pi_{n}dE\right)P_{m}\right)\right|\\
&\nonumber\le
\|\Phi(x,x)\|_{p}^{1/2}\, \|\Phi(y,y)\|_{p}^{1/2}\,
\|\left(\int_{\sigma(\Phi(x,y))}\Pi_{n}dE\right)P_{m} \|_{q}.
\end{align}

Letting $n\to\infty$ on both sides of \eqref{eq: estimates} yields, for every $m\in\mathbb{N}$:
\begin{align*}
	&\left|\rho(\Phi(x,y)\, |\Phi(x,y)|^{p-1}V_{x,y}^* P_{m})\right|=\left|\rho(\Phi(x,y)\, |\Phi(x,y)|^{p-1}P_{m}V_{x,y}^* )\right|\\	&\le
	\|\Phi(x,x)\|_{p}^{1/2}\, \|\Phi(y,y)\|_{p}^{1/2}\,
	\||\Phi(x,y)|^{p-1}P_{m}V_{x,y}^* \|_{q}.
\end{align*}

Hence, letting $m\to\infty$, and using that \[ |\Phi(x,y)|^{p-1} P_{m}V_{x,y}^* \to |\Phi(x,y)|^{p-1} V_{x,y}^*, \quad \mbox{ in $L^{q}(\rho)$ as $m\to\infty$}\] we obtain

\begin{align*}
\left|\rho\left(|\Phi(x,y)|^p\right)\right|&=\left|\rho\left(\Phi(x,y)|\Phi(x,y)|^pV_{v-y}^*\right)\right|\\&
\le
\|\Phi(x,x)\|_{p}^{1/2}\, \|\Phi(y,y)\|_{p}^{1/2}\,
\||\Phi(x,y)|^{p-1} V_{x,y}^*\|_{q}.
\end{align*}

Since $\||\Phi(x,y)|^{p-1} V_{x,y}^*\|_{q}\leq \||\Phi(x,y)|^{p-1}\|_{q}=\||\Phi(x,y)|\|_{q}^{p-1}$, 
dividing both sides by $\|\Phi(x,y)\|_{p}^{p-1}$ gives
\[
\|\Phi(x,y)\|_{p}
\le
\|\Phi(x,x)\|_{p}^{1/2}\, \|\Phi(y,y)\|_{p}^{1/2}.
\]
This completes the proof.   
    \end{proof}

Motivated by \cite[Theorem 1.3.1 (ii)]{stormer} we give the following
\begin{cor}\label{cor_3.8}
Let $\A$ be a $*$-algebra and let $\omega:\A\to L^p(\rho)$ be a positive linear map.  If $\omega(y^{*}x)\in L^p(\rho)\cap L^\infty(\rho)$ and is normal for some $x,y\in \A$, then
$$
\|\omega(y^{*}x)\|_p \le \|\omega(x^{*}x)\|_p^{1/2}\,\|\omega(y^{*}y)\|_p^{1/2}.
$$
\end{cor}\begin{proof} For $a,b\in \A$ define $\Phi_{ \omega}(x,y)=\omega(y^{*}x)$ and apply Proposition \ref{prop: num rad norm-CS} to the positive sesquilinear map $\Phi_{\omega}$.
\end{proof}
\begin{rem}\label{rem_3.9}
    Note that if $\A$ is a $C^*$-algebra with unit $\id$ and $\|\omega\|\leq1$, then for  all $a\in\A$ such that $\omega(a)\in L^p(\rho)\cap L^\infty(\rho)$ and $\omega(a)$ is normal, $$\|\omega(a)\|_p^2\leq\|\omega(a^*a)\|_p\|\omega(\id)\|_p\leq\|\omega(a^*a)\|_p^2,$$ hence we obtain certain link to Kadison-Schwarz inequality for normal elements (see, e.g., \cite[Theorem 1.3.1 (ii)]{stormer}). In fact, Corollary \ref{cor_3.8} can be viewed as an {\em opposite or symmentric version} of \cite[Theorem 1.3.1 (ii)]{stormer}, in the setting of noncommutative $L^p$-spaces, since in \cite[Theorem 1.3.1 (ii)]{stormer} it is assumed that $a$ is a normal element, whereas we assume that $\omega(a)$ is normal in Corollary \ref{cor_3.8}.% $\omega(a)=\omega(\id^*a)$ and $\omega(\id^*a)$ is assumed to be normal
\end{rem}

\subsection{The case of positive sesquilinear maps with values in ordered Banach bimodules over *-algebras}
Recall that $L^p(\rho)$ is in fact a Banach bimodule over the *-algebra $L^\infty(\rho)$. Motivated by the generalized Cauchy-Schwarz inequality for positive $L^p(\rho)$-valued sesquilinear maps, in this subsection we will study Cauchy-Schwarz inequality for positive sesquilinear maps with values in a particular class of ordered Banach bimodules over *-algebras.

We will assume that $ \YY$ is (an ordered) Banach bimodule over a  *-algebra $\YY_0$  (with $\YY_0$ equipped with a  not necessarily sub-multiplicative norm $\|\cdot\|_{\YY_0}$) with the respective cones $\KK$ and $\KK_0$, with $$\KK_0=\left\{\sum_{i=1}^Nz_i^*z_i, z_i\in\YY_0, i=1, \dots, N;  N\in\mathbb{N}\right\},$$ satisfying the following conditions:
\begin{itemize} \item[$(C1)$] Every $z \in \YY_0 $ can be written as $z=z_1-z_2+i(z_3-z_4)$, where $z_j \in \KK_0$  and $\|z_1-z_2 \|_{\YY_0}=\|z_1+z_2 \|_{\YY_0} \leq \|z \|_{\YY_0}$ and $\|z_3-z_4 \|_{\YY_0}=\|z_3+z_4 \|_{\YY_0}\leq \|z \|_{\YY_0}$.
	 
\item[$(C2)$] There exists a family $\mathcal{F} $ of sesquilinear forms $\varphi:\YY \times \YY_{0} \to\mathbb{C}$ enjoying the following properties:\begin{itemize}
	\item[i)] For every $y^+ \in \KK $ and $z^+ \in \KK_0 $, $ \varphi( y^+, z^+ ) \geq 0$ for all $\varphi \in \mathcal{F} .$
	\item[ii)]  For every $y \in \YY $ and $z \in \YY_0 $, $$|\varphi(y,z) |\leq \|y \|_{\YY} \|z \|_{\YY_0}, \quad \forall\varphi\in \mathcal{F} .$$
	\item[iii)]  For every $y \in \YY$,  the norm in $\YY$
	$$\|y\|_\YY= \sup_{\substack{z \in \YY_0 \\ \|z \|_{\YY_0}\leq 1}}  \sup_{\varphi\in \mathcal{F}} |\varphi(y,z)|.$$\end{itemize}
\end{itemize}	 
	Under these assumptions, the following proposition holds.\\

	\begin{prop}\label{prop_HilbMod}	Let $\X$ be a complex vector space,  $\mathfrak{Y}$  an ordered Banach bimodule over the  *-algebra $\mathfrak{Y}_0$ with positive closed cone $\KK$, satisfying $(C1)-(C2)$. Let $\Phi : \X \times \X \to \YY$ be a positive  sesquilinear map. %{ Suppose that $\YY$ satisfies  the following additional condition:
%			% \begin{itemize}
%				% \item[(P)]  $\ip{y}{x}\geq0$   for each $x\in \KK$ and $y\in \KK\cap\D$.
%				% \end{itemize}} 
		Then, for all $x_1,x_2 \in \X $ we have that 
		$$\|\Phi (x_1,x_2) \|_\YY  \leq 2 \|\Phi (x_1,x_1) \|_\YY^{1/2} \|\Phi (x_2,x_2) \|_\YY^{1/2} .$$ 
	\end{prop}
	
	\begin{proof}Let $ \mathcal{F}$ be the family in $(C2)$. 
		For each fixed $\varphi \in \mathcal{F} $ and $z^+ \in \KK_0$, let $\widetilde{\varphi}_{z^+} $ be the sesquilinear form on $\X \times \X$ given by $$\widetilde{\varphi}_{z^+} (x_1,x_2)=\varphi (\Phi (x_1, x_2),z^+ ),\quad\forall  x_1,x_2 \in \X.$$ Since $ \Phi$ is positive, by  $(C2)$ $i)$, we obtain that $ \widetilde{\varphi}_{z^+} $ is positive. Hence, by  $(C2)$ $ii)$, for all $x_1,x_2 \in \X $ we get that 
		  \begin{equation}\label{eq: C2} | \widetilde{\varphi}_{z^+} ( x_1,x_1)|  | \widetilde{\varphi}_{z^+} ( x_2,x_2)| \leq \|\Phi ( x_1,x_1) \|_\YY \| \Phi ( x_2,x_2) \|_\YY \|z^+ \|_{\YY_0}^{2} .\end{equation}
		Now, if we choose  $z\in \YY_0$, by  assumption, we can find $z_1^+, z_2^+,z_3^+, z_4^+ \in \KK_0$ such that $z=z_1^+ - z_2^+ +i (z_3^+ - z_4^+) $ and $\|z_j^+ \|_{\YY_0} \leq \|z\|_{\YY_0} $ for all $j \in \{1, ... , 4\} .$ Hence,  by the Cauchy-Schwarz inequality applied to the positive sesquilinear forms $\widetilde{\varphi}_{z_j^+}$ and by \eqref{eq: C2},  it is
		\begin{align*}
  | \varphi (\Phi (x_1, x_2), z) | &\leq \sum_{j=1}^{4} | \varphi (\Phi (x_1, x_2), z_j^+) | \\ 
      &\leq \sum_{j=1}^{4}  \left(\varphi (\Phi (x_1, x_1), z_j^+)\right)^{1/2} \left(\varphi (\Phi (x_2, x_2), z_j^+)\right)^{1/2} \\  &\leq \left(\sum_{j=1}^{4}  \varphi (\Phi (x_1, x_1), z_j^+)\right)^{1/2} \left(\sum_{j=1}^{4}\varphi (\Phi (x_2, x_2), z_j^+)\right)^{1/2}\\ 
      &\leq  \|\Phi (x_1,x_1) \|_\YY^{1/2} \|\Phi (x_2,x_2) \|_\YY^{1/2}(\|z_1^++z_2^+\|_{\YY_0}+\|z_3^+ +z_4^+\|_{\YY_0})\\
      &\leq 2  \|\Phi (x_1,x_1) \|_\YY^{1/2} \|\Phi (x_2,x_2) \|_\YY^{1/2} \|z \|_{\YY_0} 
		    \end{align*}       
		for all $x_1,x_2 \in \X $ and all $\varphi \in \mathcal{F}$, where the third inequality is due to the Cauchy-Schwarz applied to the inner product in $\mathbb{R}^4$. Therefore, for every $z \in \YY_0 $, with $\|z  \|_{\YY_0} \leq 1 $ and all $x_1 , x_2 \in \X$, we have 
		$$\sup_{\varphi \in \mathcal{F}} | \varphi ( \Phi (x_1,x_2),z) | \leq 2 \|\Phi (x_1,x_1) \|_\YY^{1/2} \|\Phi (x_2,x_2) \|_\YY^{1/2} .$$ 
		However, by $(C2)$ $iii)$, we deduce that for all $ x_1, x_2 \in \X$
		\begin{align*}
			\|\Phi (x_1,x_2) \|_\YY &= \sup_{\substack{z \in \YY_0 \\ \|z \|_{\YY_0}\leq 1}}  \sup_{\varphi \in \mathcal{F}} | \varphi ( \Phi(x_1,x_2),z)| \\  &\leq 2 \|\Phi (x_1,x_1) \|_\YY^{1/2} \|\Phi (x_2,x_2) \|_\YY^{1/2}.		\end{align*} 		
	\end{proof} 
\begin{ex}\label{ex: 3.9}
Let $\YY_0$ be a unital $C^*$-algebra and $\YY$ be its dual. If $\phi\in\YY$ and $z\in\YY_0$, the multiplication is defined as $$(\phi\cdot z)(w)=\phi(zw) \mbox{ and } (z\cdot \phi)(w)=\phi(wz), \forall w\in\YY_0.$$ Let $\mathfrak{K}_0$ be the natural cone in $\YY_0$  consisting of the positive elements in $\YY_0$ and the cone $\mathfrak{K}$ in $\YY$   be the dual cone i.e.,  the set consisting of all $\phi\in\YY$ such that $\phi(z)\geq0$ whenever $z\in \mathfrak{K}_0$. We can take as 
		% Let $\YY=M(\Omega)$, $\YY_0=C(\Omega)$, $\KK$ be the natural cone of positive Radon measures on $\Omega$, $\KK_0=\{f \in C(\Omega) |f \geq 0 \}$ and 
  $\mathcal{F} $ the set consisting of the single sesquilinear form $\varphi: \YY \times \YY_0\to\mathbb{C}$ given by 
		$$\varphi (\phi,z) = \phi(z^*).$$
	\end{ex}

% \section{Operator-valued case: maps into $\B(\H)$}
% classical numerical radius norm and related C-S
\section{Numerical radii norms
and Cauchy-Schwarz inequalities}
\subsection{Numerical Radius Norm on \(\B(\H)\) and Related Cauchy--Schwarz Inequalities}

Recall that the numerical radius norm on \(\B(\H)\), which we will denote  \(\|\cdot\|_{n.r.}\), is defined by

\[
\|T\|_{n.r.}=\sup_{h\in\H, \,\|h\|=1} |\langle Th, h\rangle|.
\]
For every $T\in\B(\H)$, $$
\frac{1}{2}\|T\|\leq \|T\|_{n.r.}\leq \|T\|$$ and in general inequalities are strict. 

It is easy to see that every positive \(\B(\H)\)-valued sesquilinear map satisfies the Cauchy--Schwarz inequality.  
Indeed, if

\[
\Phi : \X\times \X \longrightarrow \B(\H)
\]

is positive and sesquilinear, then for each \(h\in \H\) with \(\|h\|=1\), the map

\[
\varphi_h(x,y)=\langle \Phi(x,y)h, h\rangle, \qquad x,y\in \X,
\]

is a positive sesquilinear form.  
Applying the scalar Cauchy--Schwarz inequality to \(\varphi_h\), we obtain

\begin{align*}
    |\langle \Phi(x,y)h,h\rangle|
   &\leq \langle \Phi(x,x)h,h\rangle^{1/2}
        \langle \Phi(y,y)h,h\rangle^{1/2}\\&\leq \|\Phi(x,x)\|_{n.r.}^{1/2}
        \|\Phi(y,y)\|_{n.r.}^{1/2},
        \qquad \forall x,y\in \X.
\end{align*}

Taking the supremum over all unit vectors \(h\in \H\), we deduce

\[
\|\Phi(x,y)\|_{n.r.}
   \leq \|\Phi(x,x)\|_{n.r.}^{1/2}
        \|\Phi(y,y)\|_{n.r.}^{1/2},
        \qquad \forall x,y\in \X.
\]

\bigskip

Let now \(\M\) be a von Neumann algebra and consider the space

\[
\B\bigl(\M,(\B(\H),\|\cdot\|_{n.r.})\bigr)
\]

of all bounded linear operators from \(\M\) into \((\B(\H),\|\cdot\|_{n.r.})\), equipped with the operator norm.  
This space has an obvious cone of positive elements consisting of all those operators that map positive elements of \(\M\) into positive elements of \(\B(\H)\).

%We have the following proposition.
Now we will prove the Cauchy-Schwarz inequality for positive sesquilinear maps with values in \(\B\bigl(\M,(\B(\H),\|\cdot\|_{n.r.})\bigr)\). The proof is motivated by the one of \cite[Proposition 3.1]{BIvT2}.
\begin{prop}\label{prop: num rad norm-CS}
Let \(\Phi : \X\times \X \to \B\bigl(\M,(\B(\H),\|\cdot\|_{n.r.})\bigr)\) be a positive sesquilinear map.  
Then for all \(x,y\in \X\),

\[
\|\Phi(x,y)\|_{n.r.}\leq \|\Phi(x,x)\|_{n.r.}^{1/2}\,
                         \|\Phi(y,y)\|_{n.r.}^{1/2}.
\]

\end{prop}

\begin{proof}
Let \(\Pi=\sum_{j=1}^{N} c_{j} P_{j}\) be a simple operator-valued function in \(\M\), where $P_j$ are mutually orthogonal projections and \(\sum_{j} P_{j} = I\).  
For each \(h\in\H\) with \(\|h\|=1\), we have

\begin{align*}
    \left| \left\langle \Phi(x,y)(\Pi)h, h\right\rangle \right|
   &= \left| \sum_{j=1}^{N} c_{j}
        \left\langle \Phi(x,y)(P_{j})h, h\right\rangle \right|
  \\& \leq \sum_{j=1}^{N} |c_{j}|
        \left| \left\langle \Phi(x,y)(P_{j})h, h\right\rangle \right|\\&\leq \|\Pi\|_\M\sum_{j=1}^{N} 
        \left| \left\langle \Phi(x,y)(P_{j})h, h\right\rangle \right|,\quad\forall x,y\in \X.
\end{align*}

For each \(j\in\{1,\dots, N\}\), we apply the scalar Cauchy--Schwarz inequality to the positive sesquilinear form on $\X$
\[
\varphi_{j}(x,y)=\langle \Phi(x,y)(P_{j})h, h\rangle,\quad x,y\in \X.
\]

We obtain

\begin{align*}
 &\|\Pi\|_\M\sum_{j=1}^{N} 
        \left| \left\langle \Phi(x,y)(P_{j})h, h\right\rangle \right|\\&\leq \|\Pi\|_\M\sum_{j=1}^{N}
   \langle \Phi(x,x)(P_{j})h,h\rangle^{1/2}
   \langle \Phi(y,y)(P_{j})h,h\rangle^{1/2}\\&\leq \|\Pi\|_\M\left(\sum_{j=1}^{N}
   \langle \Phi(x,x)(P_{j})h,h\rangle\right)^{1/2}
   \left(\sum_{j=1}^{N}\langle \Phi(y,y)(P_{j})h,h\rangle\right)^{1/2}
   ,\,\,\forall x,y\in \X.
\end{align*}
the last inequality is due to the Cauchy-Schwarz inequality for the inner product in $\mathbb{R}^N$.

Summing over \(j\) and using \(\sum_{j} P_{j} = I\), we get $$\sum_{j=1}^{N}\langle \Phi(a,a)(P_{j})h,h\rangle=\langle \Phi(a,a)(I)h,h\rangle\leq\|\Phi(a,a)(I)\|_{n.r.}^{1/2}\leq\|\Phi(a,a)\|^{1/2},$$
for any $a\in\X$ and hence

\[
\left| \langle \Phi(x,y)(\Pi)h,h\rangle \right|
   \leq \|\Pi\|_\M
  \|\Phi(x,x)\|^{1/2}
   \|\Phi(y,y)\|^{1/2},\quad \forall x,y\in \X.
\]
Taking the supremum over all unit vectors \(h\in\H\), we obtain
\begin{equation}
    \label{eq: CS 1}\|\Phi(x,y)(\Pi)\|_{n.r.}
   \leq\|\Pi\|_\M \|\Phi(x,x)\|^{1/2}
        \|\Phi(y,y)\|^{1/2},\quad \forall x,y\in \X.
\end{equation}
Now choose a sequence \(\{\Pi_{n}\}\) of simple operator-valued functions converging uniformly to a unitary \(U\in \M\).  For every $\Pi_n$ \eqref{eq: CS 1} holds true, then passing to the limit as $n\to\infty$, yields
\[
\|\Phi(x,y)(U)\|_{n.r.}
   \leq \|\Phi(x,x)\|^{1/2}
        \|\Phi(y,y)\|^{1/2},\quad \forall x,y\in \X.
\]
Since any convex combination $V$ of unitaries satisfies the same inequality, and every \(T\in \M\) with \(\|T\|_\M\leq 1\) is a uniform limit of a sequence $\{V_n\}$ of such convex combinations, we conclude that
\[
\|\Phi(x,y)(T)\|_{n.r.}
   \leq \|\Phi(x,x)\|^{1/2}
        \|\Phi(y,y)\|^{1/2},
        \qquad \|T\|_\M\leq 1.
\]
Finally, taking the supremum over all \(T\in \M\) with \(\|T\|\leq 1\), we obtain the desired inequality.
\end{proof}
\begin{rem}\label{rem_4.2}
    Since $\|\cdot\|_{n.r.}\leq \|\cdot\|$ on $\B(\H)$, it is $$\B\bigl(\M,(\B
    (\H),\|\cdot\|)\bigr)\subseteq\B\bigl(\M,(\B
    (\H),\|\cdot\|_{n.r.})\bigr).$$ Now, a von Neumann algebra $\M$ can be isometrically embedded into $(\B
    (\H),\|\cdot\|)$ for a suitable Hilbert space $\H$, hence, by the above, we get $$\B(\M)\subseteq \B\bigl(\M,(\B
    (\H),\|\cdot\|_{n.r.})\bigr)$$for a suitable Hilbert space. As we will see in Section \ref{sec: appl} (Corollary \ref{moreGNS alg} and Remark \ref{rem_5.10}), by Proposition \ref{prop: num rad norm-CS}, we will be able to represent in a $B_\YY$-space every positive linear map $\omega:\A\to \B\bigl(\M,(\B
    (\H),\|\cdot\|_{n.r.})\bigr)$ on a unital *-algebra $\A$. However, by the above it follows that every positive linear map $\Omega:\A\to \B(\M)$ can thus be represented in a $B_\YY$-space. Finally, we notice that the proof of Proposition \ref{prop: num rad norm-CS} can easily be extended to the case when $(\B(\H),  \| \cdot \|_{n.r.} ) $ is replaced by a unital $C^*$-algebra $ \A $ equipped with the numerical radius norm $\nu$ defined as $$ \nu(a) = \sup \lbrace |  \theta (a) | : \theta \text{ is a state on }  \A \rbrace,\quad  a \in \A. $$
    The proof is the same as that of Proposition \ref{prop: num rad norm-CS} 
by letting states on $ \A$ play the role of the functionals $$ T \rightarrow \langle Th, h \rangle,\quad   h \in \H,\,  \| h \| = 1,\, T \in \B(\H).$$
\end{rem}

Let $\B\bigl(\M,\B(\M,(\B(\H),\|\cdot\|_{n.r.}))\bigr)$ denote the space of all bounded linear operators from \(\M\) into \(\B(\M,(\B(\H),\|\cdot\|_{n.r.}))\), equipped with the operator norm.
%%%%%%%%%%%%%%%%%da qui

% To be inserted right after the proof of Proposition 4.1
By utilizing the same techniques and procedure as in the proof of Proposition \ref{prop: num rad norm-CS}, we are now able to provide the following

\begin{cor}
Let $\Phi : \X \times \X \to \B\bigl(\M,\B(\M,(\B(\H),\|\cdot\|_{n.r.}))\bigr)$ be a positive sesquilinear map. Then, for all $x,y \in \X$ it holds that
\[
\|\Phi(x,y)\| \le \|\Phi(x,x)\|^{1/2} \, \|\Phi(y,y)\|^{1/2}.
\]
\end{cor}

\begin{proof}
Let $\Pi = \sum_{i=1}^N c_iP_i$ and $\Theta = \sum_{j=1}^M d_jQ_j$ be two simple operator-valued functions in $\M$. For each $h \in \H$ with $\|h\| = 1$, we have

\begin{align*}
    &\bigl|\langle \Phi(x,y)(\Pi)(\Theta)h, h \rangle\bigr|
\le \|\Pi\|_\M\|\Theta\|_\M\sum_{i,j} \bigl|\langle \Phi(x,y)(P_i)(Q_j)h, h \rangle\bigr|\\&\le \|\Pi\|_\M\|\Theta\|_\M\left(\langle \Phi(x,x)(I_\M)(I_\M)h, h \rangle\right)^{1/2}\left(\langle \Phi(y,y)(I_\M)(I_\M)h, h \rangle\right)^{1/2}\\&\leq\|\Pi\|_\M\|\Theta\|_\M\|\Phi(x,x)\|^{1/2}\|\Phi(y,y)\|^{1/2}.
\end{align*}

By the same arguments as in the proof of Proposition \ref{prop: num rad norm-CS}, we can deduce that
\[
\|\Phi(x,y)(\Pi)\| \le \|\Pi\|_\M\|\Phi(x,x)\|^{1/2} \, \|\Phi(y,y)\|^{1/2}
\]
Then, by repeating once again the same procedure, i.e., by approximating first unitaries by simple operator-valued functions in $\M$ and then by passing to convex combinations of unitaries and applying the Russo–Dye theorem, we finally obtain
\[
\|\Phi(x,y)\| \le \|\Phi(x,x)\|^{1/2} \, \|\Phi(y,y)\|^{1/2}.
\]
\end{proof}

\begin{rem}
By applying this procedure inductively, we can prove that the Cauchy–Schwarz inequality in the norm is satisfied for every positive sesquilinear map $\Phi$ with values in $$\B(\M,\B(\M,\B(\M,\dots (\B(\H), \|\cdot\|_{n.r.})\dots))).$$
If $\KK$ is the cone in $\B(\M,(\B(\H), \|\cdot\|_{n.r.}))$ consisting of those operators that map positive elements in $\M$ into positive elements in $\B(\H)$, then $\B(\M,(\B(\H), \|\cdot\|_{n.r.}))$ is an ordered Banach bimodule over $\M$ where the multiplication is defined in the same way as in Example \ref{ex: 3.9}. By repeating this construction inductively, we can deduce that  $\B(\M,\B(\M,\B(\M,\dots (\B(\H), \|\cdot\|_{n.r.})\dots)))$ is also an ordered Banach bimodule over $\M$.
\end{rem}

\subsection{Generalized numerical radius norm on $L^2(\rho)$ and related Cauchy-Schwarz inequality}

%\section{Third and fourth file}
% In this section we show that if $\Phi(x,y)$ is normal for some fixed $x,y\in \X$, then the Cauchy-Schwarz inequality holds for this (fixed) $\Phi(x,y)$, i.e. the norm of $\Phi(x,y)$ is less or equal the square root of the norm of $\Phi(x,x)$ times the square root of the norm of $\Phi(y,y)$.

% This holds for both the operator norm and for the non-commutative $L^2$-norm.

% Now, of course, this is not very useful for representations since here we deal with fixed $x,y \in \X$ such that $\Phi(x,y)$ is normal, but from the point of view of inequalities this result could be perhaps still interesting.

% In fact, it is an opposite version of Theorem 1.3.1 part (iii) in Størmer, as I also mention at the end of the third file.

% After all, the monograph we submit to is actually devoted to operator inequalities, so I believe that these results from the third and the fourth file could be relevant in that connection.

In this section we introduce a new norm on $L^2(\rho)$ as a generalization of the numerical radius norm on $\B(\H)$. This norm is such that every positive sesquilinear map into $L^2(\rho)$ satisfies Cauchy-Schwarz in this new norm. Hence, this allows us representations of such maps in a Banach space and not just in a quasi-Banach space.

Let $\M$ be a factor of type either I or II  on a Hilbert space $\mathcal{H}$, and let $\rho$ be a faithful semifinite trace on $\M$. Let $\tn{\cdot}_2:L^2(\rho)\to\mathbb{R}^+$ be given by
\begin{equation}\label{eq:new-norm}
\tn{F}_2
=
\sup_{\substack{
    W \in L^{\infty}(\rho)\cap L^{2}(\rho),\\ W \ge 0\\
    \|W\|_{2} \le 1,\; \|W\|_{\infty} \le 1
}}
\| W F W \|_{1}.
\end{equation}

\begin{lem}\label{lem: Lemma 3.1}
The map $\tn{\cdot}_2$ is a norm on $L^{2}(\rho)$.
\end{lem}

\begin{proof}
Let us first show that $\tn{\cdot}_2$ is well defined.
If $F \in L^2(\rho)$ and $W \in L^\infty(\rho)$, then $FW \in L^2(\rho)$.
If in addition $W \in L^{2}(\rho)$, then by the noncommutative Hölder inequality, $WFW\in L^1(\rho)$ and

\[
\|WFW\|_{1} \le \|W\|_{2}\,\|FW\|_{2}\le \|W\|_{2}\,\|F\|_{2}\|W\|_{\infty}\le \|F\|_{2}.
\]

Thus $\tn{\cdot}_{2}$ is well defined.

Homogeneity and the triangle inequality follow from the fact that $\|\cdot\|_{1}$ is a norm.
It remains to prove faithfulness.

Assume $\tn{F}_2 = 0$, then $WFW = 0$ for every finite operator
$W \in L^{\infty}(\rho)$ with $W > 0$.
Let $D \in L^{\infty}(\rho) \cap L^{2}(\rho)$ and write it as

\[
D = D_{1} - D_{2} + i(D_{3} - D_{4}),
\qquad D_{j} \ge 0, \forall j\in \{1,\dots, 4\}.
\]

Then $\|D_{j}\|_{2} \le \|D\|_{2}$ for each $j\in \{1,\dots, 4\}$.

By  Lemma \ref{projectors},  for each $j$ there exists a sequence of finite projections $\{P_{n}^{(j)}\}_n$ in $\M$
 such that

\[
\lim_{n\to\infty} \| D_{j}(I - P_{n}^{(j)}) \|_{2} = 0, \qquad \forall j\in \{1,\dots, 4\}.
\]

Moreover, by the construction in the proof of \cite[Lemma 9.8]{BDI},
each $P_{n}^{(j)}=E_{D_j}\left(\frac1n,\infty\right)$ for all $n\in \mathbb{N}$ and $j\in \{1,\dots, 4\}$ where $E_{D_j}$ is the spectral measure corresponding to $D_{j}$, hence $P_{n}^{(j)}$ commutes with ${D_{j}}^{1/2}$ for all $n\in \mathbb{N}$ and $j\in \{1,\dots, 4\}$.

Since $P_{n}^{(j)} {D_{j}}^{1/2} P_{n}^{(j)}$ is a finite positive operator, we get 
\begin{align*}
    \rho(F D_{j}P_{n}^{(j)})&=\rho(F {D_{j}}^{1/2}{D_{j}}^{1/2}P_{n}^{(j)}P_{n}^{(j)})=\rho(F ({D_{j}}^{1/2}P_{n}^{(j)})^2)\\&=\rho({D_{j}}^{1/2}P_{n}^{(j)}F {D_{j}}^{1/2}P_{n}^{(j)})\\&=\rho(P_{n}^{(j)}{D_{j}}^{1/2}P_{n}^{(j)}F P_{n}^{(j)}{D_{j}}^{1/2}P_{n}^{(j)})
\end{align*} since ${D_{j}}^{1/2}P_{n}^{(j)}={D_{j}}^{1/2}P_{n}^{(j)}P_{n}^{(j)}=P_{n}^{(j)}{D_{j}}^{1/2}P_{n}^{(j)}$, for all $n\in \mathbb{N}$ and $j\in \{1,\dots, 4\}$.  Since $WFW=0$ for every positive finite operator $W$, we get
\[
\rho(F D_{j} P_{n}^{(j)}) = 0, \quad \forall n\in\mathbb{N}, \forall j\in\{1,\dots,4\}.
\]
On the other hand, since \[
|\rho(F D_{j}(I-P_{n}^{(j)}))|
\le \|F D_{j}(I-P_{n}^{(j)})\|_{1}\leq\|F\|_{2}\, \|D_{j}(I-P_{n}^{(j)})\|_{2}\to0,\]  as $n\to\infty
$, we deduce that
 $\rho(FD_{j}) = 0$ for each $j\in\{1,\dots,4\}$.
Thus $\rho(FD)=0$ for every $D \in L^{\infty}(\rho)\cap L^{2}(\rho)$, and therefore $\rho(FF^*)=0$. Thus $F=0$ since  $\|F\|_2=0$ and $\|\cdot\|_2$ is a norm.
\end{proof}

% To be inserted after the proof of Lemma 5.1

\begin{rem}\label{rem_4.6}
If $A \in M_n(\mathbb{C})$ the space of the $n\times n$ matrices with complex entries, then for every $x \in \mathbb{C}^n$ we have
\[
|\langle Ax, x \rangle|
= \operatorname{tr}(|X^*AX|),
\]
where $X \in M_n(\mathbb{C})$ is the operator having the vector $x\in \mathbb{C}^n$ in its first column and the zero vector of $\mathbb{C}^n$ in all other ones. In this way, the norm $\tn{\cdot}_2$ can be considered as a generalization of the numerical radius norm.
\end{rem}

\begin{cor}\label{cor: 4.7} Let $\X$ be a vector space, $\Phi : \X \times \X \to (L^{2}(\rho),\tn{\cdot}_{2})$ be a positive sesquilinear map.
Then for all $x,y \in \X$,

\[
\tn{\Phi(x,y)}_{2}
\le \tn{\Phi(x,x)}_{2}^{1/2}\, \tn{\Phi(y,y)}_{2}^{1/2}.
\]

\end{cor}

\begin{proof}
Given $W \in L^{\infty}(\rho)\cap L^{2}(\rho)$ with $W \ge 0$, $\|W\|_{2},\|W\|_{\infty}\leq1$,
define

\[
\Phi_{W}(x,y) := W \Phi(x,y) W, \quad \forall x,y\in \X.
\]

Then $\Phi_{W}$ is a positive sesquilinear map with values in $L^1(\rho)$.
By \cite[Proposition 3.1, part 7]{BIvT2}, we have \begin{align*}   
\|\Phi_W(x,y)\|_{1}
&\le \|\Phi_W(x,x)\|_{1}^{1/2}\, \|\Phi_W(y,y)\|_{1}^{1/2}\\&\leq \tn{\Phi(x,x)}_{2}^{1/2}\, \tn{\Phi(y,y)}_{2}^{1/2}, \quad \forall x,y\in \X.
\end{align*}

Taking  the supremum over all such $W$ yields the claim.
\end{proof}

\medskip
\begin{rem}As we will see later in Section \ref{sec: appl}, by using Corollary \ref{cor: 4.7}, we will be able to represent in a $B_\YY$-space every positive linear map from a unital *-algebra into $L^2(\rho)$, see Corollary \ref{moreGNS alg}.
\end{rem}

Let  $\YY=\mathcal{L}^{2}(\rho)$ be the completion of $(L^{2}(\rho), \tn{\cdot}_2)$ and  $\M$ be a factor of type I or II. %and recall the norm $\tn{\cdot}_2$ on $L^{2}(\rho)$ defined earlier in \eqref{eq:new-norm}. 
Now we consider the space $
\B(\M, \mathcal{L}^{2}(\rho))$ and denote the operator norm in $\B(\M, \mathcal{L}^{2}(\rho))$ by $\|\cdot\|_{o.n.}$. %, and  
%$\B(\M, \mathcal{L}^{2}(\rho))$ denotes the space of all bounded linear operators from $\M$ into $\mathcal{L}^{2}(\rho)$, equipped with the operator norm.  
If $\KK$ is the cone in $\B(\M,\mathcal{L}^2(\rho))$ consisting of those operators that map positive elements in $\M$ into positive elements in $L^2(\rho)$, then $\B(\M,\mathcal{L}^2(\rho))$ is an ordered Banach bimodule over $\M$ where the multiplication is defined in the same way as in Example \ref{ex: 3.9}.

\begin{defn}
    A sesquilinear map $\Phi : \X \times \X \to \B(\M, \mathcal{L}^{2}(\rho))$ will be called {\em positive}  if $\Phi(x,x)(T)$ is a positive element of  
$L^{2}(\rho)$ for all $x \in \X$ whenever $T\in\M$ and $T \ge 0$.
\end{defn}
The proofs of both \cite[Proposition 3.1]{BIvT2} and of Proposition \ref{prop: num rad norm-CS} motivate the proof of the next theorem. 
\begin{thm}\label{thm: CS o.n.}    
Let $\Phi : \X \times \X \to \B(\M, \mathcal{L}^{2}(\rho))$ be a positive sesquilinear map. %, and suppose that $\Phi(x,y)(\M) \subseteq L^{\infty}(\rho) \cap L^{2}(\rho)$ for all $x,y \in \X$.  
Then for all $x,y \in \X$,
\[
\|\Phi(x,y)\|_{o.n.}
\leq
\|\Phi(x,x)\|_{o.n.}^{1/2}\,
\|\Phi(y,y)\|_{o.n.}^{1/2}.
\]
\end{thm}
\begin{proof}
Let  $W \in L^2(\rho)\cap L^\infty(\rho) $ with $W \geq 0$, $\|W\|_{2} \le 1,$ $\|W\|_{\infty} \le 1$ and $\Pi$ and $\Theta$ be two simple operator-valued functions in $\M$
$$
\Pi = \sum_{i=1}^N c_i P_i, \qquad 
\Theta = \sum_{j=1}^M d_j Q_j, 
$$
where $c_i,d_j$ are scalars for every $i\in\{1,\dots, N\}$, $j\in\{1,\dots, M\}$ and  $\{P_i\}$ and $\{Q_j\}$ are families of mutually orthogonal projections and $\sum_{i=1}^N P_i= \sum_{j=1}^M  Q_j=I$. Then, if $x,y\in\X$
\[
\left| \rho\!\left( \Pi \, W \Phi(x,y)(\Theta) W \right) \right|
= \left| \sum_{i,j} c_i d_j \, \rho\!\left( P_i W \Phi(x,y)(Q_j) W \right) \right|.
\]
Thus, \begin{align*}
    &\left| \rho\!\left( \sum_{i=1}^N c_i P_i W \Phi(x,y)\left(\sum_{j=1}^M d_j Q_j\right) W \right) \right| \\& \leq \sum_{i,j} |c_i|\,|d_j|\left| \rho\!\left(  P_i W \Phi(x,y)\left(Q_j\right) W \right) \right|\\&=\sum_{i,j} |c_i|\,|d_j|\left| \rho\!\left(  P_i^2 W \Phi(x,y)\left(Q_j\right) W \right) \right|\\&=\sum_{i,j} |c_i|\,|d_j|\left| \rho\!\left(  P_i W \Phi(x,y)\left(Q_j\right) W P_i\right) \right|.
\end{align*}
For each pair $(i,j)$ define the sesquilinear form on $\X$
\[
\psi_{ij}(x,y)
= \rho\!\left( P_i W \Phi(x,y)(Q_j) W P_i \right),\quad x,y\in\X.
\]
Since $\Phi$ is positive, each $\psi_{ij}$ is positive.  
Applying Cauchy–Schwarz to every $\psi_{ij}$ and using the fact that  $\sum_i P_i = \sum_j Q_j = I$, we obtain
\begin{align*}
    &\sum_{i,j} |c_i|\,|d_j|\left| \rho\!\left(  P_i W \Phi(x,y)\left(Q_j\right) W P_i\right) \right|\\&\leq \sum_{i,j} |c_i|\,|d_j| \left(\rho\!\left(P_i W \Phi(x,x)(Q_j)WP_i \right)\right)^{1/2}\left(\rho\!\left(P_i W \Phi(y,y)(Q_j)WP_i \right)\right)^{1/2}\\&\leq\left( \sum_{i,j} |c_i|\,|d_j|\rho\!\left(P_i W \Phi(x,x)(Q_j)WP_i \right)\right)^{1/2}\\&\qquad\qquad\cdot \left( \sum_{i,j} |c_i|\,|d_j|\rho\!\left(P_i W \Phi(y,y)(Q_j)WP_i \right)\right)^{1/2}\\&= \left( \sum_{i,j} |c_i|\,|d_j|\rho\!\left(P_i W \Phi(x,x)(Q_j)W\right)\right)^{1/2}\\&\qquad\qquad\cdot \left( \sum_{i,j} |c_i|\,|d_j|\rho\!\left(P_i W \Phi(y,y)(Q_j)W\right)\right)^{1/2}\\&\leq \left(\sum_{i=1}^N\sum_{j=1}^M  \|\Pi\|_\infty\,\|\Theta\|_\infty\rho\!\left(P_i W \Phi(x,x)(Q_j)W\right)\right)^{1/2}\\&\qquad\qquad\cdot \left( \sum_{i=1}^N\sum_{j=1}^M  \|\Pi\|_\infty\,\|\Theta\|_\infty\rho\!\left(P_i W \Phi(y,y)(Q_j)W\right)\right)^{1/2}\\&
=\|\Pi\|_\infty\,\|\Theta\|_\infty \left(\rho\!\left(W \Phi(x,x)(I)W\right)\right)^{1/2}\left(\rho\!\left(W \Phi(y,y)(I)W\right)\right)^{1/2} \\&\leq\|\Pi\|_\infty\,\|\Theta\|_\infty
\tn{\Phi(x,x)(I)}_2^{1/2}\,
\tn{\Phi(y,y)(I)}_2^{1/2}\\&\leq\|\Pi\|_\infty\,\|\Theta\|_\infty
\|\Phi(x,x)\|_{o.n.}^{1/2}\,
\|\Phi(y,y)\|_{o.n.}^{1/2}.
\end{align*}
Since this holds for all simple functions $\Pi,\Theta$ in $\M$, given an unitary operator $U \in \M$, there exists a sequence $\{\Pi_n\}_n$ of simple operator-valued functions in $\M$ such that $\Pi_n\to U$ as $n\to\infty$ in operator norm.  
Thus, for every $n\in\mathbb{N}$, by \cite[Proposition 3.4.5, Corollary 3.4.6]{Dodds} we get
\begin{align*}
   & \left| \rho\left( (U-\Pi_n) W \Phi(x,y)(\Theta) W \right) \right|
\le \rho\!\left( |(U-\Pi_n) W \Phi(x,y)(\Theta) W |\right)\\&\leq \|U-\Pi_n\|_\infty
\|W\Phi(x,y)(\Theta)W\|_{1}\to 0, \quad\mbox{ as }n\to\infty.
\end{align*}
On the other hand, by the above inequalities, we have
\begin{align*}
 \left| \rho\left( \Pi_n W \Phi(x,y)(\Theta) W \right) \right|
\leq \|\Pi_n\|_\infty\|\Theta\|_\infty
\|\Phi(x,x)\|_{o.n.}^{1/2}\,
\|\Phi(y,y)\|_{o.n.}^{1/2}
\end{align*} and  as $n\to\infty$ on both sides we obtain \begin{align*}
    \left| \rho\left( U W \Phi(x,y)(\Theta) W \right) \right|
\leq \|\Theta\|_\infty
\|\Phi(x,x)\|_{o.n.}^{1/2}\,
\|\Phi(y,y)\|_{o.n.}^{1/2}.
\end{align*}
If now $V$ is a  convex combination of unitary operators in $\M$ then we easily get also
\begin{equation}\label{eq: V conv comb}
    \left| \rho\left( V W \Phi(x,y)(\Theta) W \right) \right|
\leq \|\Theta\|_\infty
\|\Phi(x,x)\|_{o.n.}^{1/2}\,
\|\Phi(y,y)\|_{o.n.}^{1/2}.
\end{equation}

% Let $Z$ be the partial isometry from the polar decomposition of $W \Phi(x,y) W$ and let  $\{V_n\}\subset\M$ be a sequence of convex combinations of unitary operators in $\M$ such that $V_n\to Z^*$ as $n\to \infty$ in the operator norm. 
% Then,
% \[
% \left| \rho\!\left( (V_n-Z^*) W \Phi(x,y)(\Theta) W \right) \right|
% \leq\|V_n-Z^*\|_\infty\|
% W \Phi(x,y)(\Theta) W \|_1\to0
% \]
%  as $n\to\infty$.
% By \eqref{eq: V conv comb} for every $n\in\mathbb{N}$ it is \begin{equation*}
%     \left| \rho\left( V_n W \Phi(x,y)(\Theta) W \right) \right|
% \leq \|\Theta\|_\infty
% \|\Phi(x,x)\|_{o.n.}^{1/2}\,
% \|\Phi(y,y)\|_{o.n.}^{1/2},
% \end{equation*}  and as $n\to\infty$ we get
% \[
% \| W \Phi(x,y)(\Theta) W \|_1
% \le\|\Theta\|_\infty
% \|\Phi(x,x)\|_{o.n.}^{1/2} \, \|\Phi(y,y)\|_{o.n.}^{1/2}
% \]
% since $ZW\Phi(x,y)(\Theta) W=\left|W\Phi(x,y)(\Theta) W\right|$.

Let $m \in \mathbb{N}$ and $A_m \in L^{1}(\rho)\cap L^\infty(\rho)$ be such that
\[\|A_m - W \, \Phi(x,y)(\Theta) W\|_{1} < \frac{1}{2^m}
\] and let $Z_m$ be the partial isometry from the polar decomposition of $A_m$, choose a sequence
$\{V_n^{(m)}\}_{n \in \mathbb{N}} \subset \M$ of convex combinations of unitaries in $\M$ such that
\[
V_n^{(m)}\to Z_m^*  \quad \text{in the operator norm as } n\to\infty.
\]
Then,
$$
\left|\rho\left((V_n^{(m)} - Z_m^*)\, W \Phi(x,y)(\Theta) W\right)\right|
   \le \| V_n^{(m)} - Z_m^*\|_\M \, \|W\Phi(x,y)(\Theta)W\|_{1}\to0,$$ as $n\to \infty$.
 By \eqref{eq: V conv comb}, for every $n \in \mathbb{N}$, it is
\[\left|\rho\left(V_n^{(m)} W \Phi(x,y)(\Theta) W\right)\right|
   \leq \|\Theta\|_\M \, \|\Phi(x,x)\|_{o.n.}^{1/2}\, \|\Phi(y,y)\|_{o.n.}^{1/2}.\]
Letting $n \to \infty$, we obtain that 
\begin{align*}\left|\rho\left(Z_m^* W \Phi(x,y)(\Theta) W\right)\right|
  & \leq \|\Theta\|_\M \, \|\Phi(x,x)\|_{o.n.}^{1/2}\, \|\Phi(y,y)\|_{o.n.}^{1/2}.%\\&=\|\Phi(x,x)\|_{o.n.}^{1/2}\, \|\Phi(y,y)\|_{o.n.}^{1/2}.
\end{align*}
Now observe that
\begin{align*}
  &\left|\rho(|A_m|)-\left|\rho\left(Z_m^*W \Phi(x,y)(\Theta) W\right) \right|\,\right|      \\&=\left|\rho(Z_m^*A_m)-\left|\rho\left(Z_m^*W \Phi(x,y)(\Theta) W\right) \right|\,\right|\\&\leq\left|\rho(Z_m^*(A_m-W \Phi(x,y)(\Theta) W)) \right|\\&\leq\|Z_m^*\|_\infty\|A_m-W\Phi(x,y)(\Theta)W\|_1.\end{align*} by \cite[Proposition 3.4.5, Corollary 3.4.6]{Dodds}, then 
\begin{align*}
&\left|\,\| W \Phi(x,y)(\Theta)\|_{1}-\left|\rho\left(Z_m^*W \Phi(x,y)(\Theta) W\right) \right|\,\,\right|\\&=
  \left|\,\| W \Phi(x,y)(\Theta)\|_{1}-\|A_m\|_1+\|A_m\|_1-\left|\rho\left(Z_m^*W \Phi(x,y)(\Theta) W\right) \right|\,\right|\\&\leq
  \left|\,\| W \Phi(x,y)(\Theta)\|_{1}-\|A_m\|_1\,\right|+\left|\rho(|A_m|)-\left|\rho\left(Z_m^*W \Phi(x,y)(\Theta) W\right) \right|\,\right|\\&\leq\frac{1}{2^m}+\|Z_m^*\|_\infty\|A_m-W\Phi(x,y)(\Theta)W\|_1<\frac{1}{2^{m-1}}.\end{align*}

Hence,

$$\|W\Phi(x,y)(\Theta)W\|_1\leq \|\Phi(x,x)\|_{o.n.}^{1/2}\,
\|\Phi(y,y)\|_{o.n.}^{1/2}+\frac{1}{2^{m-1}}.$$
Since this holds for all $m\in\mathbb{N}$, we deduce that $$\|W\Phi(x,y)(\Theta)W\|_1\leq \|\Phi(x,x)\|_{o.n.}^{1/2}\,
\|\Phi(y,y)\|_{o.n.}^{1/2}.$$

Next, given a unitary $U \in \M$, choose a sequence
$\{\Theta_n\}$ of simple operator-valued functions in $\M$ such that
$\Theta_n \to U$ as $n\to\infty$ in the operator norm. Then,
\begin{align*}
    \|\, W \Phi(x,y)(\Theta_n - U) W \|_{1}
&\le \tn{\Phi(x,y)(\Theta_n - U)}_2\\&\leq
\|\Theta_n - U\|_\infty \, \| \Phi(x,y) \|_{o.n.}
\to 0,\quad \mbox{ as }n\to\infty.
\end{align*}
It follows that
\[
\| W \Phi(x,y)(\Theta_n) W \|_{1}
\longrightarrow
\| W \Phi(x,y)(U) W \|_{1},\quad \mbox{ as }n\to\infty.
\]
However, by the inequalities proved above,
\[
\| W \Phi(x,y)(\Theta_n) W \|_{1}
\le
\|\Theta_n\|_\infty \, \|\Phi(x,x)\|_{o.n.}^{1/2}\,
\|\Phi(y,y)\|_{o.n.}^{1/2}
\qquad \text{for all } n,
\]
hence, letting $n \to \infty$ gives
\[
\| W \Phi(x,y)(U) W \|_{1}
\le\|\Phi(x,x)\|_{o.n.}^{1/2}\,
\|\Phi(y,y)\|_{o.n.}^{1/2}.
\]
Once again, if $V$ is a convex combination of unitaries in $\M$, the same argument shows that
\begin{equation}
    \label{eq: second V}\| W \Phi(x,y)(V) W \|_{1}
\le
\|\Phi(x,x)\|_{o.n.}^{1/2}\,
\|\Phi(y,y)\|_{o.n.}^{1/2}.
\end{equation}
Now let $T \in \M$ with $\|T\|_\infty \leq 1$, and let and let  $\{V_n\}\subset\M$ be a sequence of convex combinations of unitary operators in $\M$ such that $V_n\to T$ as $n\to \infty$ in the operator norm.
Then,\begin{align*}
  &\left|\| W \Phi(x,y)(V_n) W \|_{1}
-\| W \Phi(x,y)(T) W \|_{1}\right|\\&\leq
\| W \Phi(x,y)(V_n - T) W \|_{1}
\le
\|V_n - T\|_\infty \, \|\Phi(x,y)\|_{o.n.}
\to 0,\, \mbox{as }n\to\infty.
\end{align*}
Hence, $$\| W \Phi(x,y)(V_n) W \|_{1}
\to\| W \Phi(x,y)(T) W \|_{1},\,\, \mbox{ as }n\to\infty$$
so, since the \eqref{eq: second V} holds for every $V_n$, then we can conclude that it is 
\[
\| W \Phi(x,y)(T) W \|_{1}
\le
\|\Phi(x,x)\|_{o.n.}^{1/2}\,
\|\Phi(y,y)\|_{o.n.}^{1/2}.
\]
 Since $W \in \M$ with $W \geq 0$ and both $\|W\|_\infty \leq 1$ and $\|W\|_{2} \leq 1$ has been chosen arbitrarily, taking the supremum over all such $W$ yields
\[
\tn{\Phi(x,y)(T)}_{2}
\le
\|\Phi(x,x)\|_{o.n.}^{1/2}\,
\|\Phi(y,y)\|_{o.n.}^{1/2}.
\]
Finally, taking the supremum over all $T \in \M$ with $\|T\|_\infty \leq 1$, we conclude that
\[
\|\Phi(x,y)\|_{o.n.}
\leq
\|\Phi(x,x)\|_{o.n.}^{1/2}\,
\|\Phi(y,y)\|_{o.n.}^{1/2}.
\]
This completes the proof.\end{proof}

\begin{rem}\label{rem_4.11}
    Notice that since $\tn{\cdot}_2\leq\|\cdot\|_2$, then $$\B(\M, (L^2(\rho),\|\cdot\|_2))\subseteq\B(\M, \mathcal{L}^2(\rho)).$$ As we will see later in Section \ref{sec: appl}, by Theorem \ref{thm: CS o.n.} we will be able to represent in a $B_\YY$-space every bounded positive left-invariant sesquilinear map $\Phi:\A\times\A\to\B(\M, \mathcal{L}^2(\rho))$ such that $\Phi(x,y)(\M)\subseteq L^2(\rho)$, for every $x,y\in\A$, where $\A$ is a unital quasi *-algebra. However, by the above, this means that we can represent in a $B_\YY$-space every bounded positive left-invariant sesquilinear map $\Phi:\A\times\A\to\B(\M, (L^2(\rho),\|\cdot\|_2))$.
\end{rem}

\section{Applications}\label{sec: appl}

Let $\mathcal{K}$ be a quasi $B_\YY$-space and $D(T)$ a dense subspace of $\mathcal{K}$.
A linear map $T:D(T)\to \mathcal{K}$ is said {\em $\Phi$-adjointable} if there exists a linear map $T^*$ defined on a subspace $D(T^*)\subset \mathcal{K}$ such that
$$\Phi(T\xi,\eta)= \Phi(\xi, T^*\eta), \quad \forall \xi \in D(T), \eta\in D(T^*).$$

Let $\D$ be a dense subspace of $\mathcal{K}$ and let us consider the following families of linear operators acting on $\D$:
\begin{align*}		{\LDK}&=\{T \mbox{ $\Phi$-adjointable}, D(T)=\D;\; D(T^*)\supset \D\} \\	{\Lc^\dagger(\D)}&=\{T\in \LDK: T\D\subset \D; \; T^*\D\subset \D\} \\	{\Lc^\dagger(\D)_b} &=\{T\in \Lc^\dagger(\D): \, T \mbox{ is bounded on $\D$} \}.\end{align*}
%where $\overline{W}$ denotes the closure of $W$.
The involution in $\LDK$ is defined by		$T^\dag := T^*\upharpoonright \D$, the restriction of $T^*$, the $\Phi$-adjoint of $T$, to $\D$.

The sets $\Lc^\dagger(\D)$ and ${\Lc^\dagger(\D)_b}$ are *-algebras.

\begin{rem} \label{rem_closable} If $T\in \LDK$ then $T$ is closable. We denote by $T^*$ its $\Phi$-adjoint and by $D(T^*)$ its domain.
$T^* $ is a closed operator. \end{rem}
\begin{rem} \label{rem_partialmult} As in the case of Hilbert spaces, one can prove that
$\LDK$ is a {\em partial *-algebra} \cite{ait_book}  with respect to the following operations: the usual sum $T_1 + T_2 $,
the scalar multiplication $\lambda T$, the involution $ T \mapsto T\ad := T^* \up {\D}$ and the \emph{(weak)}
partial multiplication $\mult$
defined whenever there  exists $W\in \LDK$ such that
$$\Phi(T_2 \xi,T_1\eta)= \Phi(W\xi,\eta), \quad \forall \xi,\eta \in \D.$$
Due to the density of $\D$ in $\mathcal{K}$, the element $W$, if it exists, is unique. We put $W=T_1\mult T_2$.
\end{rem}

\begin{defn} Let $\D$ be a dense subspace of $\mathcal{K}$. If $\mathcal{Y}$ is a $\dag$-invariant subset of
${\LDK}$, the {\em weak bounded commutant of} $\mathcal{Y}$ is defined to be the set
$$(\mathcal{Y},\D)_w'=\{B\in\B(\mathcal{K}): \Phi(BY\xi,\eta)=\Phi(B\xi,Y^\dag\eta),  \,\, \forall Y\in\mathcal{Y},\, \xi,\eta\in\D\}.$$
\end{defn}

% By a similar argument than in \cite[Remark ??]{BIvT1}, due to the   Cauchy-Schwarz-like inequality for positive $\YY$-valued sesquilinear maps, we can show that $X^*$ is closed.
% We prove that $X^*$ is closed. Indeed, suppose that $\{u_n\}_n$ is a sequence in $\DD(X^*)$ such that \\ \mbox{$\|u_n-u\|_\Phi \to 0$} for some $u\in \X$ and $\|X^*u_n -v\|_\Phi \to 0$ for some $v\in \X$. Clearly $\|u_n-u\|_\Phi \to 0$ is equivalent to $\Phi(u_n-u,u_n-u)\to 0$.
% Then by Lemma \ref{lem:1}, Lemma \ref{lem: 2} and Proposition \ref{prop: 1}, we get, for every $y\in \X$,
%  $$\|\Phi(u_n-u,y)\|_\YY^2 \leq 4\| \Phi(u_n -u, u_n-u)\|_\YY\|\Phi(y,y)\|_\YY\to 0.$$
% Hence, for every $z\in \DD$
% $$ \|\Phi(u_n,Xz)\|_\YY= \|\Phi(X^*u_n,z)\|_\YY\to \|\Phi(u,Xz)\|_\YY.$$
% On the other hand,
% $$ \|\Phi(X^*u_n,z)\|_\YY \to \|\Phi(v,z)\|_\YY.$$
% These relations imply that $u \in \DD(X^*)$ and $X^*u=v$. Thus $X^*$ is closed.
% Now apply this result to $X^{\dag*}$ to obtain a closed extension of $X$.

%{\color{red}\begin{rem} \label{rem_partialmult}
%$\LDK$ is also a {\em partial *-algebra} \cite{ait_book}  \ctrem{si usa? sennò cancellare} with respect to the following operations: the usual sum $X_1 + X_2 $,
%the scalar multiplication $\lambda X$, the involution $ X \mapsto X\ad := X^* \up {\DD}$ and the \emph{(weak)}
%partial multiplication $\mult$
%defined whenever there  exists $Y\in \LDK$ such that
%$$\Phi(X_2 a,X_1b)= \Phi (Ya,b), \quad \forall a,b \in \DD.$$
%The element $Y$, if it exists, is unique. We put $Y=X_1\mult X_2$.
%\end{rem}
%}

\begin{defn} \label{defn_starrepmod} 
Let $(\A,\A_0)$ be a quasi *-algebra with unit $\id$.   Let $\mathcal{D}$ be a dense subspace
of a certain $B_\YY$-space $\mathcal{K}$ with
$\YY$-valued    inner product $\ip{\cdot}{\cdot}_\mathcal{K}$.   A linear map $\pi$ from $\A$ into  ${\mathcal L}\ad(\mathcal{D},\mathcal{K})$ is called  a \emph{*-representation} of  $(\A, \A_0)$,
if the following properties are fulfilled:
\begin{itemize}
\item[(i)]  $\pi(a^*)=\pi(a)^\dagger:=\pi(a)^*\upharpoonright\mathcal{D}, \quad \forall \ a\in \A$;
\item[(ii)] for $a\in \A$ and $c\in \A_0$, $\pi(a)\mult\pi(c)$ is well-defined and \linebreak {$\pi(a)\mult \pi(c)=\pi(ac)$}.
\end{itemize}
We assume that for every *-representation $\pi$ of $(\A,\A_0)$,  $\pi(\id)={\idop_{\mathcal{D}}}$, the
identity operator on  the space $\mathcal{D}$. Let $t_\pi$ be  the graph topology defined by the seminorms $\xi\in\mathcal{D}\to \|\xi\|_\mathcal{K}+\|\pi(a)\xi\|_\mathcal{K}$, $a\in\A$, with $ \|\cdot\|_\mathcal{K}$ the    norm induced by the   inner product on $\mathcal{K}$.  \\
 The *-representation $\pi$  is said to be\begin{itemize}
% \item {\em closable} if there exists $\widetilde{\pi}$ the closure of $\pi$, defined as $\widetilde{\pi}(a)=\overline{\pi(a)}\upharpoonright{\widetilde{\D}}$ where $\widetilde{\D}$ is the completion under the graph topology $t_\pi$ defined by the seminorms $\xi\in\mathcal{D}\to \|\xi\|_\mathcal{K}+\|\pi(a)\xi\|_\mathcal{K}$, $a\in\A$, with $ \|\cdot\|_\mathcal{K}$ the    norm induced by the   inner product on $\mathcal{K}$; 
\item \emph{closed} if  $\mathcal{D}[t_\pi]$ is complete; 
\item  \emph{cyclic} if there  exists $\xi\in\mathcal{D}$ such that $\pi(\A_0)\xi$ is dense in $\mathcal{K}$ in its  norm topology. In this case $\xi$ is called {\em cyclic vector}.
\end{itemize}

\end{defn}

%\begin{defn} \label{px}
%We 	denote by $\QA$ the set of all $\YY$-valued  positive sesquilinear  maps on $\A \times \A$ that satisfy a property of invariance:
%\begin{itemize}
%\item[(I)]
%$\Phi(ac,d)=\Phi(c, a^*d), \quad \forall \ a \in \A, \ c,d \in \A_0$
%\end{itemize} and  call $\Phi\in\QA$ a $\YY$-valued {\em invariant} positive  sesquilinear  map.
%\end{defn}

%The following results given in \cite{BIvT1}, can be easily adapted to the case of $\YY$-valued sesquilinear maps due to Cauchy-Schwarz-like inequality for positive $\YY$ sesquilinear maps. We will indicate by  $\widetilde{\A}$ the completion of $\A/\mathfrak{N}_\Phi$ w.r.to $\|\cdot\|_\Phi$.
%\begin{prop}\label{prop_nonsing}
%\vspace{-1mm} Let  $\Phi$ be a  $\YY$-valued positive  sesquilinear map on $\A\times\A$.
%The following statements are equivalent:
%\begin{itemize}
%	\item[{\em (i)}]$\Lambda_\Phi(\A_0)=\A_0/\mathfrak{N}_\Phi$ is dense in  $\widetilde{\A}$.
%	\item[{\em (ii)}] If $\{a_n\}_n$ is a sequence of elements of $\A$ such that:
%	\begin{itemize}
%		\item[{\em (ii.a)}]$\Phi(a_n,c){\to} 0_\YY$, as $n \to\infty$, for every $c \in \A_0$;
%		\item[{\em (ii.b)}]$\Phi(a_n-a_m,a_n-a_m){\to} 0_\YY$, as $n,m \to\infty$;
%	\end{itemize}
%	then,  $\displaystyle \lim_{n\to \infty}\Phi(a_n,a_n) =0_\YY$.
%\end{itemize}
%\end{prop}

\begin{defn} We denote by $\IA$ the set of all $\YY$-valued  positive sesquilinear  maps on $\A \times \A$ with the following properties:
	\begin{itemize}
	\item[(i)] $\Lambda_\Phi(\A_0)=\A_0/\mathfrak{N}_\Phi$ is dense in  the completion $\widetilde{\A}$ of $\A$.
	\item[(ii)]
	 $\Phi(ac,d)=\Phi(c, a^*d), \quad \forall \ a \in \A, \ c,d \in \A_0$ (left-invariant).  \end{itemize} \end{defn}

\begin{thm}\label{thm_rep}
\vspace{-1mm} Let $(\A,\A_0)$  be a  quasi *-algebra with unit $\id$ and $\Phi$ be a $\YY$-valued left-invariant positive  sesquilinear  map on on $\A \times \A$.
The following statements are equivalent:
\begin{itemize}
\item[{\em (i)}]$\Phi\in\IA$;
\item[{\em (ii)}] there exist a (quasi) $B_\YY$-space $\mathcal{K}_\Phi$ with $\YY$-valued (quasi)-inner product  $\ip{\cdot}{\cdot}_{\mathcal{K}_\Phi}$,
%  				Banach space $\mathcal{K}_\Phi$ whose quasi norm is induced by  a $\YY$-valued inner product $\ip{\cdot}{\cdot}_{\mathcal{K}_\Phi}$, 
a dense subspace $\D_\Phi\subseteq\mathcal{K}_\Phi$ and a closed cyclic *-representation $\pi:\A\to{\mathcal L}^\dag(\D_\Phi,\mathcal{K}_\Phi)$ with cyclic vector $\xi_\Phi$ such that $$\ip{\pi(a)\xi}{\eta}_{\mathcal{K}_\Phi}=\ip{\xi}{\pi(a^*)\eta}_{\mathcal{K}_\Phi}, \quad \forall \xi,\eta\in\D_\Phi, a\in\A$$ and such that  $$\Phi(a,b)=\ip{\pi(a) \xi_\Phi}{\pi(b) \xi_\Phi}_{\mathcal{K}_\Phi}, \quad \forall  a,b\in\A.$$ %where $\xi_\Phi$ is a cyclic vector for $\pi$
\end{itemize}
\end{thm}
\begin{proof}The proof proceeds along the lines of that one of \cite[Theorem 3.2]{BIvT1}, again due to Cauchy-Schwarz-like inequality for positive $\YY$ sesquilinear maps.\end{proof}

%     The following corollary gives a result of *-representability of $\YY$-valued bounded linear positive map on $(\A,\A_0)$ (see \cite[Corollary 3.13]{BIvT1}).
%\begin{cor}\label{cor: 3.12}
%	Let $(\A[\|\cdot\|],\A_0)$ be a unital normed quasi *-algebra and $\omega$ be a $\YY$-valued bounded linear positive map on $(\A,\A_0)$ ($\omega(c^*c)\geq0$, for every $c\in\A_0$). If there exists $M>0$ such that $\|\omega(d^*c)\|_\C\leq M\|c\|\|d\|$, for all $c,d\in\A_0$, then   there
%	exists a quasi $B_\YY$-space $\X_\Phi$ whose quasi norm is induced by a $\YY$-valued quasi inner product $\ip{\cdot}{\cdot}_{\X_\Phi}$, a dense subspace $\DD_\omega\subseteq\X_\Phi$ and a closed cyclic  *--representation $\ppi_\omega$ of $(\A,\A_0)$  with domain $\DD_\omega$ and   cyclic vector $\eta_\omega$,
%	such that $$ \omega(a)=\ip{{\ppi}_\omega(a)\eta_\omega}{\eta_\omega}_{\X_\Phi}, \quad \forall \ a \in \A,$$ and $$\omega(b^*ac)=\ip{\ppi_\omega(a) \Lambda_\omega(c)}{\Lambda_\omega(b)}_{\X_\Phi},\quad\forall a\in\A,\,\,\forall b,c\in\A_0.$$ 
%	 The representation is unique up to unitary equivalence.
%\end{cor}
%}

% \begin{rem}
%     In particular, if $ \Phi:\A \times \A \to\YY$ is a bounded, left-invariant completely positive sesquilinear map  and $\omega: \A \to \YY $ is a  linear map given by $\omega(a) := \Phi (a,\id) $ for all $a \in \A ,$ then $\omega$  satisfies the conditions of Proposition \ref{prop: linear map}, indeed, for all $c,d \in \A_0$ we have that 
% 	$$\omega(d^*c)=\Phi(d^*c,\id) =\Phi(c,d)$$ 
% 	and $$ \|\Phi(c,d) \|\leq \|\Phi \|\|c \|\|d \|.$$
% \end{rem}

\begin{rem}
     By the same arguments as in \cite[Corollary 3.5]{BIvT1} 
  one can show that every $\YY$-valued, bounded, positive sesquilinear map on a unital normed quasi *-algebra belongs to $\IA$.\end{rem}

\begin{rem}Once we have the representation of    $\Phi\in\IA$ with      $(\A,\A_0)$ a quasi *-algebra with unit at hand,  the linear map on $\A$ defined as $$\omega_\Phi(a)=\Phi(a,\id),\quad \forall a\in\A$$ is also representable, by slight modifications of the result known in the literature for linear functionals on a quasi *-algebra with unit, see, e.g. \cite[Theorem 2.4.8]{FT_book}.
% ; in fact,  contrary to the case of *-algebras, a positive linear map defined
% on a quasi *-algebra $(\A,\A_0)$ with unit is not automatically representable. However, the map $\omega_\Phi$ satisfies all the three conditions for the representability (see, e.g. \cite[Definition 2.4.6]{FT_book} duly modified): $\omega_\Phi(x^*x)=\Phi(x,x)\in\KK$ for all $x\in\A_0$, $$\omega_\Phi(y^*a^*x)=\Phi(y^*a^*x,\id)=\Phi(\id ,x^*ay)=\overline{ \Phi(x^*ay,\id)}=\overline{ \omega_\Phi(x^*a^*y)}, $$  for all $a\in\A,  x,y\in\A_0$ and for every $a\in\A$
% there exists $\gamma_a>0$ such that $\|\omega(a^*x)\|_\YY\leq\gamma_a\|\omega(x^*x)\|_\YY^{1/2}$, for all $x\in\A_0$ due to Cauchy-Schwarz-like inequality for positive $\YY$ sesquilinear maps. 
Similarly to the case of a positive linear map on a *-algebra $\A_0$
without unit that does not automatically extend to a (*-representable) positive linear map on the unitization of $\A_0$, the same applies when the linear map is defined on a quasi *-algebra without unit. However,  some conditions are known for  a positive linear map on a quasi-*-algebra  without unit to be extended to a positive linear map  on a quasi *-algebra with unit, see \cite{B2013,B2015}. We refer the analysis of this case to future studies.
\end{rem}

\begin{cor}\label{moreGNS alg}\cite[Corollary 3.12]{BIvT2}  Let $\A$ be a *--algebra with unit
$\id$ and let $\omega$ be a positive  linear $\YY$-valued map on $\A$. Then, there
exists a (quasi) $B_\YY$-space $\mathcal{K}_\Phi$ whose  (quasi)-norm is induced by a $\YY$-valued   (quasi)-inner product $\ip{\cdot}{\cdot}_{\mathcal{K}_\Phi}$, a dense subspace $\D_\omega\subseteq\mathcal{K}_\Phi$ and a closed cyclic *--representation $\ppi_\omega$ of $\A$ with domain $\D_\omega$, such that $$\omega(b^*ac)=\ip{\ppi_\omega(a) \Lambda_\omega(c)}{\Lambda_\omega(b)}_{\mathcal{K}_\Phi},\quad\forall a,b,c\in\A.$$ Moreover, there exists  a  cyclic vector $\eta_\omega$,
such that $$ \omega(a)=\ip{{\ppi}_\omega(a)\eta_\omega}{\eta_\omega}_{\mathcal{K}_\Phi}, \quad \forall \ a \in \A.$$ %This
The representation is unique up to unitary equivalence. 
\end{cor}

  \begin{rem}\label{rem_5.10}
      If $\Phi$ in Theorem \ref{thm_rep} satisfies the proper Cauchy-Schwarz inequality in the norm of $\YY$, then we obtain a representation in a proper $B_\YY$-space and  not just in a  quasi $B_\YY$-space. An analogous consideration applies to Corollary \ref{moreGNS alg}.
  \end{rem}

\begin{ex}\label{ex: 2.3} Let $\rho$ be a finite trace.  Let $W\in L^\infty(\rho)$, with $W\geq 0$. Let $k\in C([0,\|W\|]\times[0,\|W\|])$ such that $k\geq0$. Then, for each $x\in[0,\|W\|]$, the function $\eta_x:[0,\|W\|]\to\mathbb{C}$ defined by $\eta_x(t)=k(x,t)$ is a continuous, positive function on $[0,\|W\|]$. Therefore, by the functional calculus, $\eta_x(W)$ defines a positive operator in $L^\infty(\rho)$. Let us define, for every $x\in[0,\|W\|]$ and $X,Y\in L^2(\rho)$,
        \begin{equation}\label{eqn_form}\varphi(X,Y)(x)=\rho(X\eta_x(W)Y^*).\end{equation} 
        Then,  $\varphi(X,Y)\in C([0,\|W\|])$ for all $X,Y\in L^2(\rho)$. Moreover, $\varphi:L^2(\rho)\times L^2(\rho)\to C([0,\|W\|]) $ is a bounded, left-invariant, positive sesquilinear map (see \cite{BIvT1}).\\
        Let    $T\in L^4(\rho)$ and consider the map $$\widetilde{\varphi}:L^2(\rho)\times L^2(\rho)\to L^2(\rho)$$
given by $\widetilde{\varphi}(X,Y)=T(\varphi(X,Y)(W))T^*$, for all $X,Y\in L^2(\rho)$, by functional calculus it is also a bounded, left-invariant positive sesquilinear map. Now, if $A,B\in L^\infty(\rho)$ are such that $A=A^*$ and $B=B^*$, let $K:=i(AB-BA)$. Then $K=K^*$. Moreover, by some calculations it can be checked that the conditions of Proposition \ref{prop: 3.6} are satisfied in this case.\\

Let us now suppose that $T\in L^{\infty}(\rho) \cap L^{4}(\rho)$ and let

\[
\Phi : L^{2}(\rho) \times L^{2}(\rho)  \to \B(\M, (\B(\H),\|\cdot\|_{n.r.}))
\]

be given by
\[
\Phi(X,Y)(S) = T\rho\left( X\eta_x(W) S\eta_x(W) Y^* \right)(W)T^*
\]
for all $S \in \M $ and $ X,Y \in L^{2}(\rho)$,
where $\B(\H)$ satisfies that $\M \subseteq \B(\H)$. Then one can check by some direct computation that $\Phi$ is a left-invariant, positive sesquilinear map. Moreover, it can also be verified that $\Phi$ is bounded, because
\[
\bigl\| \Phi(X,Y)(S) \bigr\|_{n.r.}
   \le \|T\|_\M^2\|k\|_\infty^2 \, \|X\|_2 \, \|Y\|_2\|S\|_\M\]
    for all  $S \in \M$ and $X,Y \in L^{2}(\rho)$.
Finally, for each $X,Y \in L^2(\rho)$ and $S \in \M$,
$\Phi(X,Y)(S)\in L^2(\rho)\cap L^\infty(\rho)$ and 
\[
\tn{\Phi(X,Y)(S)}_2
   \leq\|T\|_4^2\|k\|_\infty^2 \, \|X\|_2 \, \|Y\|_2\|S\|_\M,
\]so $\Phi$ may also be  regarded as a bounded, positive, left-invariant sesquilinear map from
$L^2(\rho) \times L^2(\rho) $ into $\mathcal{L}^2(\rho)$ and
such that \(
\Phi(X,Y)(\M) \subset L^2(\rho)\cap L^\infty(\rho).
\) for all $X,Y\in L^2(\rho)$.\\
If $T\in L^{4}(\rho)\setminus \cap L^{\infty}(\rho)$, then we still have that $\Phi(x,y)(S)\in L^2(\rho)$ for all $S\in\M$ and $X,Y\in L^2(\rho) $. \end{ex}

\noindent{\bf{Acknowledgements:}} GB acknowledges that this work has been carried out within the activities of Gruppo UMI Teoria dell’Appros-simazione e Applicazioni and of GNAMPA of the INdAM. SI is supported by the Ministry of Science, Technological Development and Innovations, Republic of Serbia, grant no. 451-03-66/2024-03/200029.

\bibliographystyle{amsplain}

 \end{document}